\documentclass[10pt]{article}
\usepackage{cite}
\usepackage{amsmath,amssymb,amsthm}
\usepackage{titlesec}
\usepackage{color}
\usepackage{fancyhdr}
\usepackage[margin=3cm]{geometry}
\usepackage{amsmath,amssymb,amsthm,mathrsfs,dsfont}
\usepackage{titlesec,hyperref}
\usepackage{color}
\usepackage{fancyhdr}
\titleformat{\subsection}{\it}{\thesubsection.\enspace}{1pt}{}

\makeatletter
\def\ps@pprintTitle{%
   \let\@oddhead\@empty
   \let\@evenhead\@empty
   \let\@oddfoot\@empty
  \let\@evenfoot\@oddfoot
}
\makeatother

\newtheorem{theo}{Theorem}[section]
\newtheorem{lemm}[theo]{Lemma}
\newtheorem{defi}[theo]{Definition}

\newtheorem{prop}[theo]{Proposition}
\newtheorem{rema}[theo]{Remark}
\numberwithin{equation}{section}

\def\bel{\begin{equation}\label}

\def\eeq{\end{equation}}

\newcommand{\beq}{\begin{equation}}
\newcommand{\beno}{\begin{equation*}}
\newcommand{\eeno}{\end{equation*}}

\allowdisplaybreaks
\begin{document}
\title{The estimate of lifespan and local well-posedness for the non-resistive MHD equations in homogeneous Besov spaces}
\author{
Weikui $\mbox{Ye}^1$\footnote{email: 904817751@qq.com}\quad
Wei $\mbox{Luo}^1$\footnote{email: luowei23@mail2.sysu.edu.cn}
\quad and\quad
 Zhaoyang $\mbox{Yin}^{1,2}$\footnote{email: mcsyzy@mail.sysu.edu.cn}\\
 $^1\mbox{Department}$ of Mathematics,
Sun Yat-sen University,\\ Guangzhou, 510275, China\\
$^2\mbox{Faculty}$ of Information Technology,\\ Macau University of Science and Technology, Macau, China}
\date{}
\maketitle
\begin{abstract}
 In this paper, we mainly investigate the Cauchy problem of the non-resistive MHD equation. We first establish the local existence in the homogeneous Besov space $\dot{B}^{\frac{d}{p}-1}_{p,1}\times \dot{B}^{\frac{d}{p}}_{p,1}$ with $p<\infty$, and give a lifespan $T$ of the solution which depends on the norm of the Littlewood-Paley decomposition of the initial data. Then, we prove that if the initial data $(u^n_0,b^n_0)\rightarrow (u_0,b_0)$ in $\dot{B}^{\frac{d}{p}-1}_{p,1}\times \dot{B}^{\frac{d}{p}}_{p,1}$, then the corresponding existence times $T_n\rightarrow T$, which implies that they have a common lower bound of the lifespan. Finally, we prove that the data-to-solutions map depends continuously on the initial data when $p\leq 2d$. Therefore the non-resistive MHD equation is local well-posedness in the homogeneous Besov space in the Hadamard sense. Our obtained result improves considerably the recent results in \cite{Li1,chemin1,Feffer2}.
\end{abstract}


\tableofcontents
\section{Introduction}
In this paper, we mainly investigate the Cauchy problem of the incompressible non-resistive  magnetohydrodynamic (MHD) equations:
\begin{equation}\label{sp00}
\left\{\begin{array}{lll}
u_t-\Delta u+\nabla P =b\nabla b-u\nabla u,\\
b_t+u\nabla b=b\nabla u ,\\
div~u=div~b=0, \\
(u,b)|_{t=0}:=(u_0,b_0),
\end{array}\right.
\end{equation}
where the unknowns are the vector fields $u=(u_1,u_2,...,u_d),\quad b=(b_1,b_2,...,b_d)$ and the scalar function $P$. Here, $u$ and $b$ are the velocity and magnetic, respectively, while $P$ denotes the pressure. The magnetohydrodynamic equation is a coupled system of the Navier-Stokes equation and Maxwell's equation. This model describes the interactions between the magnetic field and the fluid of moving electrically charged particles such as plasmas, liquid metals, and electrolytes. For more physical background, we refer to \cite{Biskamp,Davidson}.

The MHD equations are of great interest in mathematics and physics. Let's review some well-posedness results about the MHD equations. In the case when there is full magnetic diffusion in system \eqref{sp00}, G. Duvaut and J.-L. Lions, \cite{Duvaut} firstly established the local existence and uniqueness result in the Sobolev spaces. They also prove the global existence of strong solutions with small initial data. M. Sermange, and R. Temam\cite{Sermange} proved the global well-posedness in the Sobolev spaces with $d=2$. For the system \eqref{sp00} with magnetic diffusion, one may refer to the survey paper\cite{Lin} written by F-H. Lin and the references therein
for recent progress in this direction.

 In the physics of plasmas the magnetic diffusion is very small such that it can be neglected. In this case, the study of well-posedness will become more difficult.
We refer to \cite{Abidi,Jiu,Lin-FH,Ren,Pan} about the global existence results with initial data sufficiently close to the equilibrium. The $L^2$ decay rate was studied by R. Agapito and M. Schonbek\cite{Agapito-Schonbek}.
C. Fefferman, D. McCormicket, J. Robinson and J. Rodrigo consider the critical Sobolev space about the well-posedness of the system \eqref{sp00}. They obtained a local existence result in $\mathbb{R}^d,d=2,3$ with the initial data $(u_0, B_0)\in H^s(\mathbb{R}^d)\times H^s(\mathbb{R}^d), s>d/2$ in \cite{Feffer1} and
$(u_0, B_0)\in H^{s-1-\epsilon}(\mathbb{R}^d)\times H^s(\mathbb{R}^d), s>d/2,0<\epsilon <1$ in \cite{Feffer2}. J. Chemin, D. McCormicket, J. Robinson and J. Rodrigo \cite{chemin1}
improved Fefferman et al.'s results to the inhomogenous Besov space with the initial data
$(u_0, B_0)\in B^{\frac{d}{2}-1}_{2,1}(\mathbb{R}^d)\times B^{\frac{d}{2}}_{2,1}(\mathbb{R}^d), (d=2,3)$ and also proved the uniqueness with $d=3$. R. Wan in \cite{Wan} obtained the uniqueness with $d=2$. Recently, J. Li, W, Tan and Z. Yin in \cite{Li1} obtained the existence and uniqueness of solutions to (\ref{sp00}) with the initial data
$(u_0, B_0)\in \dot{B}^{\frac{d}{p}-1}_{p,1}(\mathbb{R}^d)\times \dot{B}^{\frac{d}{p}}_{p,1}(\mathbb{R}^d)$ $(1\leq p\leq 2d)$.

However, whether or not the solution for the non-resistive
MHD equations is local well-posedness (local existence, uniqueness and continuous dependence of
the solution) in homogeneous Besov spaces is an open problem which was proposed by Chemin et al. in \cite{chemin1}. In \cite{Li1}, J. Li, W, Tan and Z. Yin proved the local existence and uniqueness of solutions to (\ref{sp00}) in
$(u_0, B_0)\in \dot{B}^{\frac{d}{p}-1}_{p,1}(\mathbb{R}^d)\times \dot{B}^{\frac{d}{p}}_{p,1}(\mathbb{R}^d)$ $(1\leq p\leq 2d)$. But the continuous dependence of the solution for the Cauchy problem of the non-resistive MHD equations in homogeneous Besov spaces has not been proved yet. In the paper, our aim
is to solve this open problem by establishing the local well-posedness for the Cauchy problem (\ref{sp00}) in homogeneous Besov spaces. Meanwhile, we generalized the local existence's index from $1\leq p\leq 2d$\cite{Li1} to $1\leq p<\infty$.

For convenience, we transform the system (1.1) into an equivalent form of compressible type. By using $divu = divB = 0$, we have
$$u\nabla u = div(u\otimes u),\quad B\nabla B = div(B\otimes B),\quad B\nabla u = div(u\otimes B).$$
Therefore, the system (\ref{sp0}) is formally equivalent to the following equations
\begin{equation}\label{sp0}
\left\{\begin{array}{lll}
u_t-\Delta u =\mathbb{P}(b\nabla b-u\nabla u),\\
b_t+u\nabla b=b\nabla u ,\\
(u,b)|_{t=0}:=(u_0,b_0),
\end{array}\right.
\end{equation}
where $\mathbb{P}=I+\nabla (-\Delta)^{-1}div$ is the Leray project operator, and the initial data is divergence free $div~u_0=div~b_0=0$.

To solve (\ref{sp0}), the main difficulty is that the system is only partially parabolic, owing to the magnetic equation which is of hyperbolic type. This means it's hard to get the concrete expression for the lifespan $T$ (especially the lower bound of $T$), which creates the main difficulty for proving the continuous dependence. Therefore, we would like to present a general functional framework to deal with the local existence
of the solution of (\ref{sp0}) in the homogeneous Besov spaces. By obtaining the expression of the lifespan, we get the uniformly lower bound of the lifespan $T$ by a constructive way (see the key Lemma \ref{gj} below). Finally, we use the  frequency decomposition (see Theorem \ref{continuous dependence} below) to get the continuous dependence.\\
Our main theorem can be stated as follows.
\begin{theo}\label{theorem}
Let $(u_0,b_0)\in \dot{B}^{\frac{d}{p}-1}_{p,1}(\mathbb{R}^d)\times \dot{B}^{\frac{d}{p}}_{p,1}(\mathbb{R}^d) $ with $d\geq 2$. Then there exists a positive time $T$ such that\\
(1) Local existence: if $p\in [1,\infty)$, then the system (\ref{sp0}) has a local solution $(u,b)$ in $E^p_T$ with
$$E^p_T:= C([0,T];\dot{B}^{\frac{d}{p}-1}_{p,1}(\mathbb{R}^d))\cap L^1([0,T];\dot{B}^{\frac{d}{p}+1}_{p,1}(\mathbb{R}^d))\times C([0,T];\dot{B}^{\frac{d}{p}}_{p,1}(\mathbb{R}^d)).$$
(2) Uniqueness: if $p\in [1,2d]$, then the solution of (\ref{sp0}) is unique.\\
(3) Continuous dependence: if $p\in [1,2d]$, then the solution depends continuously on the initial data in $E^p_T$.
\end{theo}
\begin{rema}\label{Besov}
Comparing to \cite{Li1}, we generalized the local existence's index from $p\in [1,2d]$ to $p\in [1,\infty)$ and prove the continuous dependence with $p\in [1,2d]$ in homogeneous Besov spaces.
\end{rema}

The remainder of the paper is organized as follows. In Section 2 we introduce some useful preliminaries. In Section 3, we prove the local existence and the uniqueness of the solution to (\ref{sp0}) with the expression of local time being given. In Section 4, we firstly prove that if the initial data $(u^n_0, B^n_0)$ tends to $(u_0, B_0)$ in $ B^{\frac{d}{p}-1}_{p,1}\times B^{\frac{d}{p}}_{p,1}$, then their local existence times $T_n\rightarrow T$, which implies that they have the common
 existence time $T-\delta$ ($\delta$ is small enough). Then we use the method of frequency decomposition to obtain the continuous dependence.\\
\quad\\
\textbf{Notations: } Throughout, we donate $\dot{B}^{s}_{p,r}(\mathbb{R}^d))=\dot{B}^{s}_{p,r}$, $\|u\|_{\dot{B}^{s}_{p,r}(\mathbb{R}^d)}+\|v\|_{\dot{B}^{s}_{p,r}(\mathbb{R}^d)}=\|u,v\|_{\dot{B}^{s}_{p,r}}$ and $C([0,T];\dot{B}^{s}_{p,r}(\mathbb{R}^d))=C_T(\dot{B}^{s}_{p,r})$, $L^p([0,T];\dot{B}^{s}_{p,r}(\mathbb{R}^d))=L^p_T(\dot{B}^{s}_{p,r})$. For convenience, we donate $C_{E_0}\approx C(1+E_0+e^{E_0})$ for $C$ large enough.

\section{Preliminaries}
\par
In this section, we will recall some propositons and lemmas on the Littlewood-Paley decomposition and Besov spaces.

\begin{prop}\cite{book}
Let $\mathcal{C}$ be the annulus $\{\xi\in\mathbb{R}^d:\frac 3 4\leq|\xi|\leq\frac 8 3\}$. There exist radial functions $\chi$ and $\varphi$, valued in the interval $[0,1]$, belonging respectively to $\mathcal{D}(B(0,\frac 4 3))$ and $\mathcal{D}(\mathcal{C})$, and such that
$$ \forall\xi\in\mathbb{R}^d,\ \chi(\xi)+\sum_{j\geq 0}\varphi(2^{-j}\xi)=1, $$
$$ \forall\xi\in\mathbb{R}^d\backslash\{0\},\ \sum_{j\in\mathbb{Z}}\varphi(2^{-j}\xi)=1, $$
$$ |j-j'|\geq 2\Rightarrow\mathrm{Supp}\ \varphi(2^{-j}\cdot)\cap \mathrm{Supp}\ \varphi(2^{-j'}\cdot)=\emptyset, $$
$$ j\geq 1\Rightarrow\mathrm{Supp}\ \chi(\cdot)\cap \mathrm{Supp}\ \varphi(2^{-j}\cdot)=\emptyset. $$
The set $\widetilde{\mathcal{C}}=B(0,\frac 2 3)+\mathcal{C}$ is an annulus, and we have
$$ |j-j'|\geq 5\Rightarrow 2^{j}\mathcal{C}\cap 2^{j'}\widetilde{\mathcal{C}}=\emptyset. $$
Further, we have
$$ \forall\xi\in\mathbb{R}^d,\ \frac 1 2\leq\chi^2(\xi)+\sum_{j\geq 0}\varphi^2(2^{-j}\xi)\leq 1, $$
$$ \forall\xi\in\mathbb{R}^d\backslash\{0\},\ \frac 1 2\leq\sum_{j\in\mathbb{Z}}\varphi^2(2^{-j}\xi)\leq 1. $$
\end{prop}

\begin{defi}\cite{book}
Denote $\mathcal{F}$ by the Fourier transform and $\mathcal{F}^{-1}$ by its inverse.
Let $u$ be a tempered distribution in $\mathcal{S}'(\mathbb{R}^d)$. For all $j\in\mathbb{Z}$, define
$$
\Delta_j u=0\,\ \text{if}\,\ j\leq -2,\quad
\Delta_{-1} u=\mathcal{F}^{-1}(\chi\mathcal{F}u),\quad
\Delta_j u=\mathcal{F}^{-1}(\varphi(2^{-j}\cdot)\mathcal{F}u)\,\ \text{if}\,\ j\geq 0,\quad
S_j u=\sum_{j'<j}\Delta_{j'}u.
$$
Then the Littlewood-Paley decomposition is given as follows:
$$ u=\sum_{j\in\mathbb{Z}}\Delta_j u \quad \text{in}\ \mathcal{S}'(\mathbb{R}^d). $$

Let $s\in\mathbb{R},\ 1\leq p,r\leq\infty.$ The nonhomogeneous Besov space $B^s_{p,r}(\mathbb{R}^d)$ is defined by
$$ B^s_{p,r}=B^s_{p,r}(\mathbb{R}^d)=\{u\in S'(\mathbb{R}^d):\|u\|_{B^s_{p,r}(\mathbb{R}^d)}=\Big\|(2^{js}\|\Delta_j u\|_{L^p})_j \Big\|_{l^r(\mathbb{Z})}<\infty\}. $$

Similarly, we can define the homogeneous Besov space.
$$\dot{B}^{s}_{p,r}=\dot{B}^{s}_{p,r}(\mathbb{R}^d):=\{u\in S'_h(\mathbb{R}^d) | \|u\|_{\dot{B}^{s}_{p,r}}:=\|2^{sj}\|\dot{\Delta}_ju\|_{L^p(\mathbb{S}^d)}\|_{l^r}\leq\infty\},$$
where the Littlewood-Paley operator $\dot{\Delta}_j$ is defined by
$$\dot{\Delta}_j u=\mathcal{F}^{-1}(\varphi(2^{-j}\cdot)\mathcal{F}u)\,\ \text{if}\,\ j\in\mathbb{Z}.$$
\end{defi}

\begin{lemm}\label{compact}
Let $s\in (-\frac{d}{p'},\frac{d}{p}]$ ($s=\frac{d}{p},r=1$). Assume $f^n$ is uniformly bounded in $\dot{B}^s_{p,r}\cap\dot{B}^{-\delta}_{\infty,\infty}(\forall\delta>0)$ or $\dot{B}^s_{p,r}\cap L^{\infty}$ . Then $\varphi f^n$ is bound in $\dot{B}^{s}_{p,r}\cap\dot{B}^{s-\epsilon_1}_{p,r}$ $(0<\epsilon_1<s+\frac{d}{p'})$, and the map $f^n\mapsto \varphi f^n$ is compact in $\dot{B}^{s-\epsilon}_{p,r}$ $(0<\epsilon<\epsilon_1)$, where $\varphi\in S(\mathbb{R}^d)$.
\end{lemm}
\begin{proof}
The proof is based on Theorems 2.93-2.94 in \cite{book}, we omit it here.
\end{proof}

\begin{lemm}\label{zhibiao}\cite{miao}
Let $s_1,s_2\leq \frac{d}{p}$ and $s_1+s_2>d\max\{0,\frac{2}{p}-1\}$. Assume $f\in\dot{B}^{s_1}_{p,1}$ and $g\in\dot{B}^{s_2}_{p,1}$. Then there holds
$$\|fg\|_{\dot{B}^{s_1+s_2-\frac{d}{p}}_{p,\infty}}\leq C\|f\|_{\dot{B}^{s_1}_{p,\infty}}\|g\|_{\dot{B}^{s_2}_{p,1}}.$$
\end{lemm}

\begin{defi}\cite{book}
Let $s\in\mathbb{R},1\leq p,q,r\leq\infty$ and $T\in (0,\infty].$ The functional space $\widetilde{L}^q_T(\dot{B}^{s}_{p,r})$ is defined as the set of all the distributions $f(t)$ satisfying
$\|f\|_{\widetilde{L}^q_T(\dot{B}^{s}_{p,r})}:=\|(2^{ks}\|\dot{\Delta}_kf(t)\|_{L^q_TL^p})_k\|_{l^r}<\infty .$
\end{defi}
By Minkowski's inequality, it is easy to find that
$$\|f\|_{\widetilde{L}^q_T(\dot{B}^{s}_{p,r})}\leq \|f\|_{L^q_T(\dot{B}^{s}_{p,r})}\quad q\leq r,\quad\quad\quad \|f\|_{\widetilde{L}^q_T(\dot{B}^{s}_{p,r})}\geq \|f\|_{L^q_T(\dot{B}^{s}_{p,r})}\quad q\geq r.$$

Finally, we state some useful results about the heat equation and the transport equation
\begin{equation}\label{s1cuchong}
\left\{\begin{array}{l}
    u_t+\Delta u=G,\ x\in\mathbb{R}^d,\ t>0, \\
    u(0,x)=u_0(x),
\end{array}\right.
\end{equation}
\begin{equation}\label{s1}
\left\{\begin{array}{l}
    f_t+v\cdot\nabla f=g,\ x\in\mathbb{R}^d,\ t>0, \\
    f(0,x)=f_0(x),
\end{array}\right.
\end{equation}
which are crucial to the proof of our main theorem later.

\begin{lemm}\label{heat}\cite{Danchin}
Let $s\in\mathbb{R}, 1\leq q,q_1,p,r\leq\infty$ with $q_1\leq q$. Assume $u_0$ in $\dot{B}^s_{p,r}$, and $G$ in $\widetilde{L}^{q_1}_T(\dot{B}^s_{p,r})$. Then (\ref{s1cuchong}) has a unique solution $u$ in $\widetilde{L}^{q}_T(\dot{B}^{s+\frac{2}{q}}_{p,r})$ satisfying
$$ \|u\|_{\widetilde{L}^{q}_T(\dot{B}^{s+\frac{2}{q}}_{p,r})}\leq C_1\Big(\|u_0\|_{\dot{B}^s_{p,r}}+\|G\|_{\widetilde{L}^{q_1}_T\dot{B}^{s+\frac{2}{q_1}-2}_{p,r}}\Big). $$
In particular, if $q_1=r=1$, by Minkowski's inequality we have
$$ \|u\|_{L^{\infty}_T(\dot{B}^s_{p,1})\cap L^{2}_T(\dot{B}^{s+1}_{p,1})\cap L^{1}_T(\dot{B}^{s+2}_{p,1})}\leq C_1\Big(\|u_0\|_{\dot{B}^s_{p,1}}+\|G\|_{L^{1}_T\dot{B}^{s}_{p,1}}\Big). $$
\end{lemm}

\begin{lemm}\label{priori estimate}\cite{book}
Let $s\in [\max\{-\frac{d}{p},-\frac{d}{p'}\},\frac{d}{p}+1]$ ($s=1+\frac{1}{p},r=1$; $s=\max\{-\frac{d}{p},-\frac{d}{p'}\},r=\infty$).
There exists a constant $C$ such that for all solutions $f\in L^{\infty}([0,T];B^s_{p,r})$ of \eqref{s1} with initial data $f_0$ in $\dot{B}^s_{p,r}$, and $g$ in $L^1([0,T];\dot{B}^s_{p,r})$, we have, for a.e. $t\in[0,T]$,
$$ \|f(t)\|_{\dot{B}^s_{p,r}}\leq e^{C_2 V(t)}\Big(\|f_0\|_{\dot{B}^s_{p,r}}+\int_0^t e^{-C_2 V(t')}\|g(t')\|_{\dot{B}^s_{p,r}}dt'\Big), $$
where $V'(t)=\|\nabla v\|_{\dot{B}^{\frac{d}{p}}_{p,r}\cap L^{\infty}}$(if $s=1+\frac{1}{p},r=1$, $V'(t)=\|\nabla v\|_{\dot{B}^{\frac{d}{p}}_{p,1}}$).
\end{lemm}

\begin{rema}\cite{book}
If $div~v=0$, we can get the same result with a better indicator: $\max\{-\frac{d}{p},-\frac{d}{p'}\}-1<s<\frac{d}{p}+1$(or $s=\max\{-\frac{d}{p},-\frac{d}{p'}\}-1,r=\infty$).
\end{rema}

\begin{lemm}\label{Besov}
Let $s\in(\max\{1-\frac{d}{p},1-\frac{d}{p'}\},\frac{d}{p}]$ ($s=\frac{d}{p},r=1$), $f_0\in\dot{B}^{s}_{p,r}$, $g_0\in L^1_T(\dot{B}^{s}_{p,r})$ and $\nabla v\in L^1_T(\dot{B}^{\frac{d}{p}}_{p,1})$. If $v(t,x)$ satisfies one of the following conditions $(\rho>1)$:\\
1) when $s>1$, $v\in L^{\rho}_T(L^{\infty}\cap\dot{B}^{\frac{d}{p}}_{p,\infty})$;\\
2) when $s=1$, $v\in L^{\rho}_T(L^{\infty}\cap\dot{B}^{\frac{d}{p}}_{p,r'})$; \\
3) when $s<1$ and $1\leq p\leq 2$, $v\in L^{\rho}_T(L^{\infty}\cap\dot{B}^{\frac{d}{p}}_{p,\infty}\cap\dot{B}^{\frac{d}{p'}}_{p',r'})$; \\
4) when $s<1$ and $p\geq 2$, $v\in L^{\rho}_T(L^{\infty}\cap\dot{B}^{\frac{d}{p}}_{p,r'}) $. \\
Then (\ref{s1}) has a unique solution $f\in C_T(\dot{B}^{s}_{p,1})$ with $r<\infty$ ($f\in C_{Tw}(\dot{B}^{s}_{p,\infty})$ with $r=\infty$).
\end{lemm}

\begin{proof}
Without loss of generality, we only give the proof with $s=\frac{d}{p},r=1$, other cases are similar.

Firstly, we smooth out the data:
$$f^n_0:=S_nf_0,\quad g^n:=\rho_n*_tS_ng,\quad v^n:=\rho_n*_tS_nv.$$
Hence, the function
$$f^n(t,x)=f^n_0(\psi^{-1}_t(x))+\int_{0}^{t}g^n(s,\psi_s(\psi^{-1}_t(x)))ds$$
is a solution to
$$\frac{d}{dt}f^n(t,\psi_t(x))=g^n(t,\psi_t(x)).$$

Further, by Theorem 3.14 in \cite{book}, we have
\begin{align}\label{s2}
\|f^n\|_{\dot{B}^{\frac{d}{p}}_{p,1}}\leq Ce^{\int_{0}^{T}\|v\|_{\dot{B}^{1+\frac{d}{2}}_{p,1}}ds}(\|f^n_0\|_{\dot{B}^{\frac{d}{p}}_{p,1}}+\int_{0}^{T}\|g^n\|_{\dot{B}^{\frac{d}{p}}_{p,1}}ds).
\end{align}
Then, setting $\bar{f}^n:=f^n-\int_{0}^{t}g^n(s)ds=-\int_{0}^{t}v^n\nabla f^nds$, by the Bony decomposition, we get

\begin{equation}\label{s3}
\|v^n\nabla f^n\|_{L^{\rho}\dot{B}^{\frac{d}{p}-1}_{p,\infty}}\leq\left\{\begin{array}{lll}
\|v^n\|_{L^{\rho}(L^{\infty}\cap\dot{B}^{\frac{d}{p}}_{p,\infty})}\|f^n\|_{L^{\infty}\dot{B}^{\frac{d}{p}}_{p,1}},\quad\quad\quad\quad \text{when } s>1,\\
\|v^n\|_{L^{\rho}(L^{\infty}\cap\dot{B}^{\frac{d}{p}}_{p,\infty})}\|f^n\|_{L^{\infty}\dot{B}^{\frac{d}{p}}_{p,1}},\quad\quad\quad\quad \text{when } s=1,\\
\|v^n\|_{L^{\rho}(L^{\infty}\cap\dot{B}^{\frac{d}{p}}_{p,\infty}\cap\dot{B}^{\frac{d}{p'}}_{p',\infty})}\|f^n\|_{L^{\infty}\dot{B}^{\frac{d}{p}}_{p,1}},\quad \text{when } s<1,1\leq p\leq 2 ,\\
\|v^n\|_{L^{\rho}(L^{\infty}\cap\dot{B}^{\frac{d}{p}}_{p,\infty})}\|f^n\|_{L^{\infty}\dot{B}^{\frac{d}{p}}_{p,1}},\quad\quad\quad\quad \text{when } s<1, p\geq 2.\\
\end{array}\right.
\end{equation}
This implies that $\bar{f}^n$ is uniformly bounded in $C^{\beta}_T(\dot{B}^{\frac{d}{p}-1}_{p,\infty})\cap L^{\infty}_T(\dot{B}^{\frac{d}{p}}_{p,1})$. Lemma \ref{compact} guarantees that the map
$$\bar{f}^n\rightarrow \varphi \bar{f}^n,\quad\forall\varphi\in C^{\infty}_0 $$
is compact in $\dot{B}^{\frac{d}{p}-1}_{p,\infty}$. Combining Ascoli's theorem and Cantor's diagonal process thus ensures that
$$\varphi\bar{f}^n\rightarrow \varphi\bar{f}\quad in\quad C_T(\dot{B}^{\frac{d}{p}-1}_{p,\infty}).$$
By the Fatou property, we have $\varphi\bar{f}^n\rightharpoonup \varphi\bar{f}\in L^{\infty}_T(\dot{B}^{\frac{d}{p}}_{p,1})$. By interpolation,  we get
$$\varphi\bar{f}^n\rightarrow \varphi\bar{f}\quad in\quad C_T(\dot{B}^{\frac{d}{p}-\epsilon}_{p,1}),\quad 0<\epsilon <1.$$

Finally, applying the above results we can pass the limit in the weak sense:
$$f:=\lim_{n\rightarrow\infty}f^n=\bar{f}+\int_{0}^{t}g(s)ds.$$
It is easily to deduce that $f(t,x)$ is a solution of (\ref{s1}) and $f\in C_T(\dot{B}^{\frac{d}{p}}_{p,1})$ (For more details see Theorem 3.19 in \cite{book}). This completes the proof.
\end{proof}

\begin{rema}\label{Besov}
If  $div~v=0$, we can get the same result with a better indicator: $\max\{-\frac{d}{p},-\frac{d}{p'}\}<s<\frac{d}{p}$ ($s=\frac{d}{p},r=1$). The proof is similar to Lemma \ref{Besov}, we omit the detail here.
\end{rema}

\begin{defi}\cite{book}
Let $a>0$, $\mu(r)$ be a continue non-zero and non-decreasing function from $[0,a]$ to $\mathbb{R}^+$, $\mu(0)=0$. We say that $\mu$ is an Osgood modulus of continuity if
$$\int_{0}^{a}\frac{1}{\mu(r)}dr=+\infty.$$
\end{defi}

\begin{lemm}\label{osgood}\cite{book}
Let $\rho$ be a measurable function from $[0,T]$ to $[0,a]$, $\gamma$ a locally integrable function from $[0,T]$ to $\mathbb{R}^+$, and $\mu$ be an Osgood modulus of continuity. If for some $\rho_0\geq 0$,
$$\rho(t)\leq \rho_0+\int_{0}^{t}\gamma(s)\mu(\rho(s))ds\quad for\quad a.e.\quad t\in[0,T],$$
then we have
\begin{equation}
-M(\rho(t))+M(\rho_0)\leq \int_{0}^{t}\gamma(s)ds\quad with\quad M(x)=\int_{x}^{a}\frac{dr}{\mu(r)}.
\end{equation}
\end{lemm}

For example, if $\mu(r)=r$, we obtain the Gronwall inequality:
$$\rho(t)\leq \rho_0e^{\int_{0}^{t}\gamma(s)ds},\quad M(x)=lna-lnx.$$
If $\mu(r)=rln(e+c/r)$, it's easy to check that it is still an Osgood modulus of continuity. Then we have
$$\rho(t)\leq \rho_0\frac{ce^{e^{\int_{0}^{t}\gamma(s)ds}}}{c-\rho_0(e^{\int_{0}^{t}\gamma(s)ds}-e)},
\quad -M(\rho(t))+M(\rho_0)\geq ln[\frac{ln(e+\frac{c}{\rho_0})}{ln(e+\frac{c}{\rho(t)})}].$$
Since $\gamma$ is locally integrable, we deduce that if $\rho_0$ small enough such that $\rho_0\leq\frac{c}{2(e^{\int_{0}^{t}\gamma(s)ds}-e)}$, then
$$\rho(t)\leq 2\rho_0e^{e^{\int_{0}^{t}\gamma(s)ds}}.$$
\par

\section{Local existence and uniqueness}

We divide the proof of local existence and uniqueness into 4 steps:\\

\textbf{Step 1: An iterative scheme.}

Set $(u^n_0,b^n_0):=(\dot{S}_nu_0,\dot{S}_nb_0)$ and define the first term $(u^0,b^0):=(e^{t\Delta}u_0,e^{t\Delta}b_0)$. Then we introduce a sequence $(u^n,b^n)$ with the initial data $(u^n_0,b^n_0)$ by solving the following linear transport and heat conductive equations:
\begin{equation}\label{sp1}
\left\{\begin{array}{lll}
u^{n+1}_t-\Delta u^{n+1}=\mathbb{P}(b^{n}\nabla b^{n}-u^{n}\nabla u^{n}),\\
b^{n+1}_t+u^{n}\nabla b^{n+1}=b^{n}\nabla u^{n} ,\\
(u^n_0,b^n_0):=(\dot{S}_nu_0,\dot{S}_nb_0),
\end{array}\right.
\end{equation}
where $\dot{S}_ng:=\sum_{k<n}\dot{\Delta}_k g$, it makes sense in Besov spaces when $s<\frac{d}{p}$ or $s=\frac{d}{p},r=1$.
\quad\\

\textbf{Step 2: Uniform estimates.}

Taking advantage of Lemmas \ref{heat}-\ref{priori estimate}, we shall bound the approximating sequences in $E^p_T$. Now we claim that there exists some  $T$ independent of $n$ such that the solutions $(u^n,b^n)$ satisfy the following inequalities :
$$(H_1):\quad \|b^{n}\|_{L^{\infty}_T (\dot{B}^{\frac{d}{p}}_{p,1})}+\|u^{n}\|_{L^{\infty}_T (\dot{B}^{\frac{d}{p}-1}_{p,1})}\leq 6E_0, $$
$$(H_2):\quad \|u^{n}\|_{A_T}\leq 2a,\quad A_T:={L^{2}_T(\dot{B}^{\frac{d}{p}}_{p,1})\cap L^{1}_T(\dot{B}^{\frac{d}{p}+1}_{p,1}}),$$
where $E_0:=\|b_0\|_{\dot{B}^{\frac{d}{p}}_{p,1}}+\|u_0\|_{\dot{B}^{\frac{d}{p}-1}_{p,1}}$. Now we suppose that $a$ is small enough such that ($a$ will be determined later):
\begin{align}\label{xxiugailsp2}
a\leq \min\{\sqrt{\frac{E_0}{4C_1}},c\},
\end{align}
where $c$ is any positive real number satisfying $c\leq \frac{1}{12},e^{C_2c}\leq\frac{3}{2},4cC_1\leq\frac{1}{2}$. Suppose that $T$ satisfies that
\begin{align}\label{lsp2}
 C_1E^2_0T\leq\frac{1}{72}a,\quad 36C_1E_0T\leq 1,  \quad \|e^{t\Delta}u_0\|_{A_T}\leq a,
\end{align}
where $C_1$ and $C_2$ are the constants in Lemmas \ref{heat}-\ref{priori estimate}. (Indeed, we should take $C_1$ and $C_2$ more large as we need.)

It's easy to check that $(H_1)-(H_2)$ hold true for $n=0$. Now we will show that if $(H_1)-(H_2)$ hold true for $n$, then they hold true for $n+1$. In fact, by (\ref{xxiugailsp2})-(\ref{lsp2}) and Lemmas \ref{heat}-\ref{priori estimate}, we have
\begin{align}\label{lsp3}
\|u^{n+1}\|_{A_T} &\leq \|e^{t\Delta}u_0\|_{A_T}+\|\mathbb{P}div(-u^{n}\otimes u^{n}+b^{n}\otimes b^{n})\|_{L^{1}_T(\dot{B}^{\frac{d}{p}-1}_{p,1})}\notag \\
&\leq a+C_14a^2+36C_1E^2_0T \leq 2a,
\end{align}
\begin{align}\label{lsp4}
\|u^{n+1}\|_{L^{\infty}_T\dot{B}^{\frac{d}{p}-1}_{p,1}}&\leq \|e^{t\Delta}u_0\|_{L^{\infty}_T\dot{B}^{\frac{d}{p}-1}_{p,1}}+\|\mathbb{P}div(-u^{n}\otimes u^{n}+b^{n}\otimes b^{n})\|_{L^{1}_T\dot{B}^{\frac{d}{p}-1}_{p,1}}\notag \\
&\leq \|u_0\|_{\dot{B}^{\frac{d}{p}-1}_{p,1}}+\|\mathbb{P}div(-u^{n}\otimes u^{n}+b^{n}\otimes b^{n})\|_{L^{1}_T\dot{B}^{\frac{d}{p}-1}_{p,1}}\notag \\
&\leq E_0+C_14a^2+36C_1E^2_0T\leq 3E_0.
\end{align}
and
\begin{align}\label{lsp5}
\|b^{n+1}\|_{L^{\infty}_T(\dot{B}^{\frac{d}{p}}_{p,1})}&\leq e^{C_2a}(\|b_0\|_{\dot{B}^{\frac{d}{p}}_{p,1}}+\|div(u^{n}\otimes b^{n})\|_{L^{1}_T(\dot{B}^{\frac{d}{p}}_{p,1})}\notag \\
&\leq e^{C_2a}(E_0+12aE_0)\notag \\
&\leq 3E_0.
\end{align}
This implies $(H_1)-(H_2)$ hold true for $n+1$.

Finally, we have to obtain the relationship between the existence time $T$ and the initial data via (\ref{lsp2}). It is easy to deduce  that
$$T\leq T_0:=\min\{\frac{a}{72C_1E^2_0},\frac{1}{36C_1E_0}\}.$$
Now we turn to study the condition $\|e^{t\Delta}u_0\|_{A_T}\leq a$ of (\ref{lsp2}). For this purpose, we have to classify the initial data.\\
(1) For $\|u_0\|_{\dot{B}^{\frac{d}{p}-1}_{p,1}}\leq \bar{c}=\min\{\frac{1}{4C_1},c\}$, we let $a:=\min\{\sqrt{\frac{E_0}{4C_1}},c\}$, which implies \eqref{xxiugailsp2}.

Then we have
$$\|e^{t\Delta}u_0\|_{A_T}\leq \|u_0\|_{\dot{B}^{\frac{d}{p}-1}_{p,1}}\leq\min\{\sqrt{\frac{\|u_0\|_{\dot{B}^{\frac{d}{p}-1}_{p,1}}}{4C_1}},c\}\leq a. $$
(2) For $\|u_0\|_{\dot{B}^{\frac{d}{p}-1}_{p,1}}> \bar{c}=\min\{\frac{1}{4C_1},c\}$, we let $a:=\min\{\sqrt{\frac{\bar{c}}{4C_1}},c\}\leq \min\{\sqrt{\frac{E_0}{4C_1}},c\}$, which also implies \eqref{xxiugailsp2}.

Since $u_0\in\dot{B}^{\frac{d}{p}-1}_{p,1}$, there exists an integer $j_0$ such that ($j_0$ may not be unique):
\begin{align}\label{lsp5.5}
\sum_{|j|\geq j_0}\|\dot{\Delta}_ju_0\|_{L^p}2^{(\frac{d}{p}-1)j}< \frac{a}{4}.
\end{align}
Defining that $T_1:=\frac{a}{4}\frac{1}{2^{2j_0}\|u_0\|_{\dot{B}^{\frac{d}{p}-1}_{p,1}}}$ and $T_2:=\frac{a^2}{4^2}\frac{1}{2^{2j_0}\|u_0\|^2_{\dot{B}^{\frac{d}{p}-1}_{p,1}}}$, we get
\begin{align}\label{lsp6}
&\|e^{t\Delta}u_0\|_{L^{1}_{T_1}(\dot{B}^{\frac{d}{p}+1}_{p,1})}                                    \notag \\
&\leq \sum_{|j|\leq j_0}\int_{0}^{T_1}\|e^{t\Delta}\dot{\Delta}_ju_0\|_{L^p}2^{(\frac{d}{p}+1)j}dt+\sum_{|j|> j_0}\int_{0}^{T_1}e^{-t2^{2j}}\|\dot{\Delta}_ju_0\|_{L^p}2^{(\frac{d}{p}+1)j}dt         \notag \\
&\leq 2^{2j_0}\sum_{|j|\leq j_0}\int_{0}^{T_1}\|\dot{\Delta}_ju_0\|_{L^p}2^{(\frac{d}{p}-1)j}dt+\sum_{|j|> j_0}\int_{0}^{T_1}e^{-t2^{2j}}\|\dot{\Delta}_ju_0\|_{L^p}2^{(\frac{d}{p}+1)j}dt         \notag \\
&\leq 2^{2j_0}T_1\|u_0\|_{\dot{B}^{\frac{d}{p}-1}_{p,1}}+\sum_{|j|> j_0}(1-e^{-T_22^{2j}})\|\dot{\Delta}_ju_0\|_{L^p}2^{(\frac{d}{p}-1)j}                    \notag \\
&\leq 2^{2j_0}T_1\|u_0\|_{\dot{B}^{\frac{d}{p}-1}_{p,1}}+\sum_{|j|> j_0}\|\dot{\Delta}_ju_0\|_{L^p}2^{(\frac{d}{p}-1)j}\leq \frac{1}{2} a,
\end{align}
and
\begin{align}\label{lsp7}
&\|e^{t\Delta}u_0\|_{L^{2}_{T_2}(\dot{B}^{\frac{d}{p}}_{p,1})}                                    \notag \\
&\leq \sum_{|j|\leq j_0}[\int_{0}^{T_2}\|e^{t\Delta}\dot{\Delta}_ju_0\|^2_{L^p}dt]^{\frac{1}{2}}2^{\frac{d}{p}j}+\sum_{|j|> j_0}[\int_{0}^{T_2}(e^{-t2^{2j}}\|\dot{\Delta}_ju_0\|_{L^p})^2dt]^{\frac{1}{2}} 2^{\frac{d}{p}j}     \notag\\
&\leq 2^{j_0}T^{\frac{1}{2}}_2\|u_0\|_{\dot{B}^{\frac{d}{p}-1}}+\sum_{|j|> j_0}(1-e^{-T_22^{2j}})^{\frac{1}{2}}\|\dot{\Delta}_ju_0\|_{L^p}2^{(\frac{d}{p}-1)j}            \notag \\
&\leq 2^{j_0}T^{\frac{1}{2}}_2\|u_0\|_{\dot{B}^{\frac{d}{p}-1}}+\sum_{|j|> j_0}\|\dot{\Delta}_ju_0\|_{L^p}2^{(\frac{d}{p}-1)j}\leq \frac{1}{2} a.
\end{align}
Letting $T=\min\{T_0,T_1,T_2\}$,  we get
$$\|e^{t\Delta}u_0\|_{A_T}\leq a.$$

Finally, if we choose $T$ to satisfy that
\begin{equation}\label{lsp8}
  T=\begin{cases}
    T_0, & \|u_0\|_{\dot{B}^{\frac{d}{p}-1}_{p,1}}\leq \frac{1}{4C_1},   \\
    \min\{T_0, T_1, T_2\}, & \|u_0\|_{\dot{B}^{\frac{d}{p}-1}_{p,1}}>\frac{1}{4C_1},
  \end{cases}
\end{equation}
then (\ref{lsp2}) holds true. For this $T$, we have the approximate sequence $(u^n,b^n)$ is uniformly bounded in $E^p_T$.

\begin{rema}\label{j0}
By (\ref{lsp8}), we know that if the initial data is small, the local existence time $T$ depends only on $E_0$. However, for large initial data, the local existence time $T$ depends on both $E_0$ and the $j_0$ which satisfies (\ref{lsp5.5}).
\end{rema}
\quad\\
\textbf{Step 3: Existence of a solution.}\\

This step is similar to the process of \cite{book,Li1,miao}, we also use the compactness argument in Besov spaces for the approximate sequence $(u^n,b^n)$ to get some solution $(u,b)$ of (\ref{sp0}). Since $(u^n,b^n)$ is uniformly bounded in $E^p_T$, the interpolation inequality yields that $u^{n+1}$ is also uniformly bounded in $L^{q}_T(\dot{B}^{\frac{d}{2}-1+\frac{2}{q}}_{p,1})$ for $1\leq q\leq\infty$. Then, by Lemma \ref{heat}-\ref{priori estimate}, after some calculations, we can easily get that (for fixed $0<\epsilon<\frac{d}{p}$):
$$\partial_t u^{n+1} \text{is uniformly bounded in } L^{\frac{2}{2-\epsilon}}_T(\dot{B}^{\frac{d}{p}-1-\epsilon}_{p,1}+\dot{B}^{\frac{d}{p}-1}_{p,1}),$$
$$\partial_t b^{n+1} \text{is uniformly bounded in } L^{2}_T(\dot{B}^{\frac{d}{p}-1}_{p,1}).$$

Let $\{\chi_j\}_{j\in\mathbb{N}}$ be a sequence of smooth functions with value in $[0,1]$ supported in the ball $B(0,j+1)$ and equal to 1 on $B(0,j)$. The above argument ensures that $u^{n+1}$ is uniformly bounded in $C^{\sigma(\epsilon)}_T(\dot{B}^{\frac{d}{p}-1-\epsilon}_{p,1}+\dot{B}^{\frac{d}{p}-1}_{p,1})\cap C_T(\dot{B}^{\frac{d}{p}-1}_{p,1})$ ($\sigma(\epsilon)>0$ for fixed $\epsilon>0$ small enough), and $b^{n+1}$ is uniformly bounded in $C^{\frac{1}{2}}_T(\dot{B}^{\frac{d}{p}-1}_{p,1})\cap C_T(\dot{B}^{\frac{d}{p}}_{p,1})$. Then by Lemma \ref{compact} with $\epsilon_1=2\epsilon$ ($d\geq 2$), since the embedding $\dot{B}^{\frac{d}{p}-1-2\epsilon}_{p,1}\cap \dot{B}^{\frac{d}{p}-1}_{p,1}\hookrightarrow \dot{B}^{\frac{d}{p}-1-\epsilon}_{p,1}$ and $\dot{B}^{\frac{d}{p}-1-2\epsilon}_{p,1}\cap \dot{B}^{\frac{d}{p}}_{p,1}\hookrightarrow B^{\frac{d}{p}-1}_{p,1}$ are locally compact,
by applying Ascoli's theorem and Cantor's diagonal process, there exist some functions $(u_j,b_j)$ such that for any $j\in\mathbb{N}$, $\chi_ju^n$ tends to $u_j$, and $\chi_jb^n$ tends to $b_j$. As $\chi_j\chi_{j+1}=\chi_j$, we have $u_j=\chi_ju_{j+1}$ and $b_j=\chi_jb_{j+1}$. From that, we can easily deduce that there exists $(u,b)$ such that for all $\chi\in D(R^d)$,
\begin{equation}\label{step1}
  \left\{\begin{array}{l}
    \chi u^n\rightarrow \chi u \quad in\quad C_T(\dot{B}^{\frac{d}{p}-1-\epsilon}_{p,1}),  \\
    \chi b^n\rightarrow \chi b \quad in\quad C_T(\dot{B}^{\frac{d}{p}-1}_{p,1}),
  \end{array}\right.
\end{equation}
as n tends to $\infty$ (up to a subsequence). By interpolation, we have
\begin{equation}\label{step2}
  \left\{\begin{array}{l}
    \chi u^n\rightarrow \chi u \quad in\quad L^1_T(\dot{B}^{\frac{d}{p}+1-\delta}_{p,1}),\quad  0<\epsilon <1+\epsilon,   \\
    \chi b^n\rightarrow \chi b \quad in\quad C_T(\dot{B}^{\frac{d}{p}-\delta}_{p,1}),\quad    0<\delta <1.
  \end{array}\right.
\end{equation}
Note that $(u^n,b^n)$ is uniformly bounded in $E^p_T$. By the Fatou property, we readily get
$$(u,b)\in (\widetilde{L}^{\infty}(\dot{B}^{\frac{d}{p}-1}_{p,1})\cap L^{1}(\dot{B}^{\frac{d}{p}+1}_{p,1}))^d  \times (L^{\infty}(\dot{B}^{\frac{d}{p}}_{p,1}))^d.$$

Finally, it is a routine process to verify that $(u,b)$ satisfies the system (\ref{sp0}). Following the argment of Theorem 3.19 in \cite{book}, we have $(u,b)\in E^p_T$.\\

\textbf{Step 4: Uniqueness.}

The proof of the uniqueness of (\ref{sp0}) is similar to \cite{Li1} with $p\leq 2d$, we omit it here.\\

\section{Continuous dependence}

Before proving the continuous dependence of solutions to (\ref{sp0}), firstly we need to prove that let $T$ be a lifespan corresponding to the initial data $u_0$ by (\ref{lsp8}), if $(u^n_0,b^n_0)$ tends to $(u_0,b_0)$ in $\dot{B}^{\frac{d}{p}-1}_{p,1}\times \dot{B}^{\frac{d}{p}}_{p,1}$, then there exists a lifespan $T^n$ corresponding to $(u^n_0,b^n_0)$ such that $T^n\rightarrow T$. This implies that $T-\delta$ (for some small $\delta$) is a common lifespan both for $u^n$ and $u$ when $n$ is sufficiently large. We first give a useful lemma:
\begin{lemm}\label{gj}
Let $(u_0,b_0)\in \dot{B}^{\frac{d}{p}-1}_{p,1}\times\dot{B}^{\frac{d}{p}}_{p,1}$ be the initial data of (\ref{sp0}) with $p\leq 2d$, if there exists another initial data $(u^n_0,b^n_0)\in \dot{B}^{\frac{d}{p}-1}_{p,1}\times\dot{B}^{\frac{d}{p}}_{p,1}$ such that $\|u^n_0-u_0\|_{\dot{B}^{\frac{d}{p}-1}_{p,1}},\|b^n_0-b_0\|_{\dot{B}^{\frac{d}{p}}_{p,1}}\rightarrow 0\quad (n\rightarrow\infty)$, then we can construct a lifespan $T^n$ corresponding to $(u^n_0,b^n_0)$ such that
$$T^n\rightarrow T,\quad\quad n\rightarrow\infty ,$$
where the lifespan $T$ correspondsto $(u_0,b_0)$.
\end{lemm}
\begin{proof}
By virtue of Remark \ref{j0}, we only consider the large initial data. Thus, we need to prove that $T^n\rightarrow T$, when $\|u_0\|_{\dot{B}^{\frac{d}{p}-1}_{p,1}}>\frac{1}{4C_1}$. For convenience, we write down the definitions of $T_0, T_1, T_2$ here:
$$T_0=\min\{\frac{a}{72C_1E^2_0},\frac{1}{36C_1E_0}\},\quad T_1=\frac{a}{4}\frac{1}{2^{2j_0}\|u_0\|_{\dot{B}^{\frac{d}{p}-1}_{p,1}}}, \quad T_2=\frac{a^2}{4^2}\frac{1}{2^{2j_0}\|u_0\|^2_{\dot{B}^{\frac{d}{p}-1}}},$$
where $j_0$ is a fixed integer such that
$$\sum_{|j|\geq j_0}\|\dot{\Delta}_j u_0\|_{L^p}2^{(\frac{d}{p}-1)j}< \frac{a}{4}.$$
Since $u_0\in \dot{B}^{\frac{d}{p}-1}_{p,1}$, we can suppose that $j_0$ is the smallest integer such that the above inequality holds true.
Since $E^n_0\rightarrow E_0$, it follows that $T^n_0\rightarrow T_0$. In order to prove that $T^n_1 \rightarrow  T_1$ and $T^n_2 \rightarrow T_2$, it is sufficient to show that there exists a corresponding sequence $j^n_0$ satisfying
$$\sum_{|j|\geq j^n_0}\|\dot{\Delta}_j u^n_0\|_{L^p}2^{(\frac{d}{p}-1)j}< \frac{a}{4},$$
and $j^n_0\rightarrow j_0$.


For any $0<\epsilon<\frac{a}{4}$, there exists $N_{\epsilon}$ such that for $n\geq N_{\epsilon}$ we have
$$\|u^n_0-u_0\|_{\dot{B}^{\frac{d}{p}-1}_{p,1}}\leq \epsilon.$$
For this $\epsilon$, we define that $j^{\epsilon}_0$ is the smallest integer such that
$$\sum_{|j|\geq j^{\epsilon}_0}\|\dot{\Delta}_j u_0\|_{L^p}2^{(\frac{d}{p}-1)j}< \frac{a}{4}-\epsilon.$$
By the definition of $j_0$, we have $j_0\leq j^{\epsilon}_0$.

Replacing $\epsilon $ by $\frac{\epsilon}{m}~(m\in \mathbb{N}^+)$, we can find $N_{\frac{\epsilon}{m}}$ such that for $n\geq N_{\frac{\epsilon}{m}}$,
$$ \|u^n_0-u_0\|_{\dot{B}^{\frac{d}{p}-1}_{p,1}}\leq \frac{\epsilon}{m}.$$
For this $\frac{\epsilon}{m}$, we define that $j^{\frac{\epsilon}{m}}_0$ is the smallest integer such that
$$\sum_{|j|\geq j^{\frac{\epsilon}{m}}_0}\|\dot{\Delta}_j u_0\|_{L^p}2^{(\frac{d}{p}-1)j}< \frac{a}{4}-\frac{\epsilon}{m}.$$
Since $\frac{a}{4}-\frac{\epsilon}{m}>\frac{a}{4}-\frac{\epsilon}{m-1}$, it follows that
$$j_0\leq j^{\frac{\epsilon}{m}}_0\leq j^{\frac{\epsilon}{m-1}}_0.$$

Now letting $\bar{j}^m_0:=j^{\frac{\epsilon}{m}}_0,m=1,2,3,...$, we deduce that
\begin{align}\label{jianqie}
\sum_{|j|\geq\bar{j}^m_0}\|\dot{\Delta_j}u^n_0\|_{L^p}2^{(\frac{d}{p}-1)j}\leq \|u^n_0-u_0\|_{\dot{B}^{\frac{d}{p}-1}_{p,1}}+\sum_{|j|>\bar{j}^m_0}\|\dot{\Delta}_j u_0\|_{L^p}2^{(\frac{d}{p}-1)j}
<\frac{\epsilon}{m}+\frac{a}{4}-\frac{\epsilon}{m}=\frac{a}{4},\quad n\geq N_{\frac{\epsilon}{m}}.
\end{align}
Since $\bar{j}^m_0$ is a monotone and bounded sequence, we deduce that $\bar{j}^m_0\rightarrow \bar{j}_0$ $(m\rightarrow\infty)$ for some integer $\bar{j}_0\geq j_0$. For any $0<\bar{\epsilon}<1$ there exists $N$ such that if $m\geq N$
$$ |\bar{j}^m_0-\bar{j}_0|\leq\bar{\epsilon}<1,$$
Note that $\bar{j}^m_0,\bar{j}_0$ are integers, we deduce that $\bar{j}_0=\bar{j}^m_0$ when $m\geq N$ and $\bar{j}_0$ is the smallest integer such that
$$\sum_{|j|\geq\bar{j}_0}\|\dot{\Delta}_ju_0\|_{L^p}2^{(\frac{d}{p}-1)j}=\sum_{|j|\geq\bar{j}^m_0}\|\dot{\Delta}_ju_0\|_{L^p}2^{(\frac{d}{p}-1)j}< \frac{a}{4}-\frac{\epsilon}{m}.$$

We claim that $\bar{j}_0=j_0$. Otherwise, if $\bar{j}_0>j_0$, we deduce from the above inequality that
$$\sum_{|j|\geq j_0}\|\dot{\Delta}_ju_0\|_{L^p}2^{(\frac{d}{p}-1)j}\geq \frac{a}{4}-\frac{\epsilon}{m},\quad \forall m\geq N.$$
Since the left hand-side of the above inequality is independent of $m$, we have $$\sum_{|j|\geq j_0}\|\dot{\Delta}_ju_0\|_{L^p}2^{(\frac{d}{p}-1)j}\geq \frac{a}{4}.$$
This contradicts the definition of $j_0$. So we have $\bar{j}^m_0\rightarrow \bar{j}_0=j_0$ $(m\rightarrow\infty)$.

Finally, taking $\epsilon=\frac{a}{8}<\frac{a}{4}$, we can construct a sequence $\{j^n_0\}$ by $\{\bar{j}^m_0\}$ when $n\geq N_{\epsilon}$:
\begin{equation}\label{lammaequ1-1-5}
  j^n_0:=\left\{\begin{array}{l}
    \bar{j}^1_0, \quad N_{\epsilon}\leq n<N_{\frac{\epsilon}{2}},\\
    \bar{j}^2_0, \quad N_{\frac{\epsilon}{2}}\leq n< N_{\frac{\epsilon}{3}},\\
    ... ...\\
    \bar{j}^m_0, \quad N_{\frac{\epsilon}{m}}\leq n< N_{\frac{\epsilon}{m+1}},\\
    ... ...\\
  \end{array}\right.
\end{equation}
By virtue of \eqref{jianqie}, one can check that
$$\sum_{|j|\geq j^n_0}\|\dot{\Delta}_ju^n_0\|_{L^p}2^{(\frac{d}{p}-1)j}< \frac{a}{4}.$$
Using the monotone bounded theorem, one can prove that $j^n_0\rightarrow j_0(n\rightarrow\infty)$.
Therefore, we have
$$T^n_1\rightarrow T_1, T^n_2\rightarrow T_2  \quad\Longrightarrow T^n\rightarrow T,\quad n\rightarrow\infty .$$
This completes the proof of the lemma.
\end{proof}

\begin{rema}
The sequence $\bar{j}^m_0$ we construct in the proof of Lemma \ref{gj} is only a subsequence since $m\neq n$ but depends on $n$. And one can obtain a subsequence $T^{n_m}$ of $T^n$ such that $T^{n_m}\to T$. This is much weaker than previous one. Therefore , we have to construct the $j^n$ by \eqref{lammaequ1-1-5}.
\end{rema}

\begin{rema}
By Lemma \ref{gj}, letting $T$ be the lifespan time of $(u^{\infty},b^{\infty})$, then we can define a $T^n$ corresponding with $(u^n,b^n)$ such that $T^n\rightarrow T,n\rightarrow\infty$. That is, for fixed any small $\delta >0$, there exists an integer $N$, when $n\geq N$, we have
$$|T^n-T|< \delta .$$

Thus, we can consider $T_n:=\min\{T^{n},T\}$ as the common lifespan both for $(u^{\infty},b^{\infty})$ and $(u^{n},b^{n})$. Then we still have
$$T_n\rightarrow T,\quad n\rightarrow\infty .$$
Roughly, we can choose $T-\delta$ as the common lifespan both for $(u^{\infty},b^{\infty})$ and $(u^{n},b^{n})$, which is independent of $n$.
\end{rema}
Now we begin to prove the continuous dependence.
\begin{theo}\label{continuous dependence}
Let $p\leq 2d$. Assume that $(u^n,b^n)_{n\in\mathbb{N}}$ be the solution to the system (\ref{sp0}) with the initial data $(u^n_0,b^n_0)_{n\in\mathbb{N}}$. If $(u^n_0,b^n_0)$ tends to $(u^{\infty}_0,b^{\infty}_0)$ in $\dot{B}^{\frac{d}{p}-1}_{p,1}\times\dot{B}^{\frac{d}{p}}_{p,1}$, then there exists a positive ${T}$ independent of $n$ such that $(u^n,b^n)$ tends to $(u^{\infty},b^{\infty})$ in $C_{T}(\dot{B}^{\frac{d}{p}-1}_{p,1})\cap L^{1}_{T}(\dot{B}^{\frac{d}{p}+1}_{p,1})\times C_{T}(\dot{B}^{\frac{d}{p}}_{p,1})$.
\end{theo}

\begin{proof}
Our aim is to estimate $\|u^n-u^{\infty}\|_{L^{\infty}_{T}(\dot{B}^{\frac{d}{p}-1}_{p,1})\cap L^{1}_{T}(\dot{B}^{\frac{d}{p}+1}_{p,1})}$ and $\|b^n-b^{\infty}\|_{L^{\infty}_{T}(\dot{B}^{\frac{d}{p}}_{p,1})}$ when $n\rightarrow\infty$. Note that
\begin{equation}\label{frequence}
  \left\{\begin{array}{l}
    \|u^n-u^{\infty}\|_{{L^{\infty}_{T}(\dot{B}^{\frac{d}{p}-1}_{p,1})\cap L^{1}_{{T}}(\dot{B}^{\frac{d}{p}+1}_{p,1})}}\\
    \leq\|u^n-u^n_j\|_{{L^{\infty}_{{T}}(\dot{B}^{\frac{d}{p}-1}_{p,1})\cap L^{1}_{{T}}(\dot{B}^{\frac{d}{p}+1}_{p,1})}}+\|u^n_j-u^{\infty}_j\|_{{L^{\infty}_{{T}}(\dot{B}^{\frac{d}{p}-1}_{p,1})\cap L^{1}_{{T}}(\dot{B}^{\frac{d}{p}+1}_{p,1})}}
    +\|u^{\infty}_j-u^{\infty}\|_{{L^{\infty}_{{T}}(\dot{B}^{\frac{d}{p}-1}_{p,1})\cap L^{1}_{{T}}(\dot{B}^{\frac{d}{p}+1}_{p,1})}},\\
    \|b^n-b^{\infty}\|_{L^{\infty}_{T}(\dot{B}^{\frac{d}{p}}_{p,1})}\\
    \leq\|b^n-b^n_j\|_{L^{\infty}_{T}(\dot{B}^{\frac{d}{p}}_{p,1})}
    +\|b^n_j-b^{\infty}_j\|_{L^{\infty}_{T}(\dot{B}^{\frac{d}{p}}_{p,1})}
    +\|b^{\infty}_j-b^{\infty}\|_{L^{\infty}_{T}(\dot{B}^{\frac{d}{p}}_{p,1})},
  \end{array}\right.
\end{equation}
where
$$(u^n,b^n) \text{ corresponds to the initial data } (u^n_0,b^n_0), \quad n\in \mathbb{N}\cup {\infty},$$
$$(u^{n}_j,b^{n}_j) \text{ corresponds to the initial data } (\dot{S}_ju^n_0,\dot{S}_jb^n_0),  \quad n\in \mathbb{N}\cup {\infty}.$$

By Lemma \ref{gj}, we find that $T-\delta$ (we still write it as $T$) is the common lifespan for $(u^{n},b^{n})$, $(u^{n}_j,b^{n}_j)$, $(u^{\infty},b^{\infty})$ and $(u^{\infty}_j,b^{\infty}_j)$ when $n,j$ are large enough. By the argument as in Step 2, since $(u^n_0,b^n_0)\rightarrow(u^{\infty}_0,b^{\infty}_0)$ and $(\dot{S}_ju^n_0,\dot{S}_jb^n_0)\rightarrow(u^{n}_0,b^{n}_0)$ in $\dot{B}^{\frac{d}{p}-1}_{p,1}\times\dot{B}^{\frac{d}{p}}_{p,1}$, it follows that for any large $n$ and $j$,
\begin{align}\label{jie}
\|u^n,u^n_j\|_{L^{\infty}_{T}(\dot{B}^{\frac{d}{p}-1}_{p,1})},\quad\|b^n,b^n_j\|_{L^{\infty}_{T}(\dot{B}^{\frac{d}{p}}_{p,1})}\leq C_{E_0},\quad\|u^n\|_{L^{p}_{T}(\dot{B}^{\frac{d}{p}}_{p,1})\cap L^{1}_{T}(\dot{B}^{\frac{d}{p}+1}_{p,1})}\leq 2a\leq \frac{1}{4C_1},
\end{align}
where $E^n_0:=\|u^{n}_0\|_{\dot{B}^{\frac{d}{p}-1}_{p,1}}+\|b^{n}_0\|_{\dot{B}^{\frac{d}{p}}_{p,1}}$, $a$ is a small quantity satisfying \eqref{xxiugailsp2}.
For any $t\in [0.T]$, we now divide the estimations of \eqref{frequence} into 4 steps.\\
\quad\\
\textbf{Step 1. Estimate }$\|u^n_j-u^{\infty}_j\|_{{L^{\infty}_{T}(\dot{B}^{\frac{d}{p}-1}_{p,1})\cap L^{1}_{T}(\dot{B}^{\frac{d}{p}+1}_{p,1})}}$ \textbf{ and } $\|b^n_j-b^{\infty}_j\|_{L^{\infty}_{T}(\dot{B}^{\frac{d}{p}}_{p,1})}$ \textbf{for fixed $j$}.

Recall the equations of $(u^n_j,b^n_j)$, $n\in\mathbb{N}\cup \{\infty\}$:
\begin{equation}\label{spnj1}
\left\{\begin{array}{lll}
u^{n}_{jt}-\Delta u^{n}_j=\mathbb{P}(b^{n}_j\nabla b^{n}_j+u^{n}_j\nabla u^{n}_j),\\
b^{n}_{jt}+u^{n}_j\nabla b^{n}_j=b^{n}_j\nabla u^{n}_j ,\\
(u^n_0,b^n_0):=(\dot{S}_ju^n_0,\dot{S}_jb^n_0).
\end{array}\right.
\end{equation}

Multiplying both sides of the first equation in \eqref{spnj1} by $\eta$ ($\eta$ is determined later) and applying Lemmas \ref{heat}-\ref{priori estimate} to \eqref{spnj1}, we have
\begin{align}\label{spnju}
&\eta(\|u^n_j\|_{\dot{B}^{\frac{d}{p}}_{p,1}}+\|u^n_j\|_{L^2_t(\dot{B}^{\frac{d}{p}+1}_{p,1})}+\|u^n_j\|_{L^1_t(\dot{B}^{\frac{d}{2}+2}_{p,1})})\notag\\
&\leq \eta\|\dot{S}_ju^{n}_0\|_{\dot{B}^{\frac{d}{p}}_{p,1}}
+\eta\int_{0}^{t}\|u^n_j\|_{\dot{B}^{\frac{d}{p}}_{p,1}}\|u^n_j\|_{\dot{B}^{\frac{d}{p}+1}_{p,1}}
+\|b^n_j\|_{\dot{B}^{\frac{d}{p}}_{p,1}}\|b^n_j\|_{\dot{B}^{\frac{d}{p}+1}_{p,1}}ds\notag\\
&\leq 2^j\eta\|\dot{S}_ju^{n}_0\|_{\dot{B}^{\frac{d}{p}-1}_{p,1}}
+\eta\int_{0}^{t}\|u^n_j\|_{\dot{B}^{\frac{d}{p}}_{p,1}}\|u^n_j\|_{\dot{B}^{\frac{d}{p}+1}_{p,1}}
+\|b^n_j\|_{\dot{B}^{\frac{d}{p}}_{p,1}}\|b^n_j\|_{\dot{B}^{\frac{d}{p}+1}_{p,1}}ds
\end{align}
and
\begin{align}\label{spnjb}
\|b^n_j\|_{\dot{B}^{\frac{d}{p}+1}_{p,1}}
&\leq \|\dot{S}_jb^{n}_0\|_{\dot{B}^{\frac{d}{p}+1}_{p,1}}
+C_{E_0}\int_{0}^{t}\|b^n_j\|_{\dot{B}^{\frac{d}{p}+1}_{p,1}}\|u^n_j\|_{\dot{B}^{\frac{d}{p}+1}_{p,1}}
+\|b^n_j\|_{\dot{B}^{\frac{d}{p}}_{p,1}}\|u^n_j\|_{\dot{B}^{\frac{d}{p}+2}_{p,1}}ds\notag\\
&\leq 2^j\|\dot{S}_jb^{n}_0\|_{\dot{B}^{\frac{d}{p}}_{p,1}}
+C'_{E_0}\|u^n_j\|_{L^1_T(\dot{B}^{\frac{d}{p}+2}_{p,1})}+C\int_{0}^{t}\|u^n_j\|_{\dot{B}^{\frac{d}{p}+1}_{p,1}}\|b^n_j\|_{\dot{B}^{\frac{d}{p}+1}_{p,1}}ds,
\end{align}
where we used the fact that $\|\dot{S}_jg\|_{\dot{B}^{\frac{d}{p}}_{p,1}}\leq C2^{m}\|\dot{S}_jg\|_{\dot{B}^{\frac{d}{p}-m}_{p,1}},\quad m>0 .$

Then setting $\eta>4C'_{E_0}$, combining \eqref{spnju}, \eqref{spnjb} and the Gronwall inequality, we thus have
\begin{align}\label{spnjbu}
&\frac{\eta}{2}(\|u^n_j\|_{\dot{B}^{\frac{d}{p}}_{p,1}}+\|u^n_j\|_{L^2_t(\dot{B}^{\frac{d}{p}+1}_{p,1})}+\|u^n_j\|_{L^1_t(\dot{B}^{\frac{d}{p}+2}_{p,1})})
+\|b^n_j\|_{\dot{B}^{\frac{d}{p}+1}_{p,1}}\notag\\
&\leq 2^j(\|b^{n}_0\|_{\dot{B}^{\frac{d}{p}}_{p,1}}+\|\dot{S}_ju^{n}_0\|_{\dot{B}^{\frac{d}{p}-1}_{p,1}})\\
&\quad +C_{E_0}\int_{0}^{t}\|u^n_j\|_{\dot{B}^{\frac{d}{p}}_{p,1}}\|u^n_j\|_{\dot{B}^{\frac{d}{p}+1}_{p,1}}
+\|b^n_j\|_{\dot{B}^{\frac{d}{p}}_{p,1}}\|b^n_j\|_{\dot{B}^{\frac{d}{p}+1}_{p,1}}
+\|u^n_j\|_{\dot{B}^{\frac{d}{p}+1}_{p,1}}\|b^n_j\|_{\dot{B}^{\frac{d}{p}+1}_{p,1}}ds\notag\\
&\leq C_{E_0,j}[\|b^{n}_0\|_{\dot{B}^{\frac{d}{p}}_{p,1}}+\|u^{n}_0\|_{\dot{B}^{\frac{d}{p}-1}_{p,1}}]\\
&\quad +\int_{0}^{t}\|u^n_j\|_{\dot{B}^{\frac{d}{p}+1}_{p,1}}\|u^n_j\|_{\dot{B}^{\frac{d}{p}}_{p,1}}
+\|b^n_j\|_{\dot{B}^{\frac{d}{p}}_{p,1}}\|b^n_j\|_{\dot{B}^{\frac{d}{p}+1}_{p,1}}
+\|u^n_j\|_{\dot{B}^{\frac{d}{p}+1}_{p,1}}\|b^n_j\|_{\dot{B}^{\frac{d}{p}+1}_{p,1}}ds\notag\\
&\leq C'_{E_0,j}(\|b^{n}_0\|_{\dot{B}^{\frac{d}{p}}_{p,1}}+\|u^{n}_0\|_{\dot{B}^{\frac{d}{p}-1}_{p,1}}),
\end{align}
which along with the Gronwall inequality leads to
\begin{align}
\frac{\eta}{2}(\|u^n_j\|_{\dot{B}^{\frac{d}{p}}_{p,1}}+\|u^n_j\|_{L^2_t\dot{B}^{\frac{d}{p}+1}_{p,1}}+\|u^n_j\|_{L^1_t\dot{B}^{\frac{d}{p}+2}_{p,1}})+\|b^n_j\|_{\dot{B}^{\frac{d}{p}+1}_{p,1}} \leq C'_{E_0,j}(\|b^{n}_0\|_{\dot{B}^{\frac{d}{p}}_{p,1}}+\|u^{n}_0\|_{\dot{B}^{\frac{d}{p}-1}_{p,1}}).
\end{align}

For fixed $j$, letting $\delta^n u=u^n_j-u^{\infty}_j$ and $\delta^n b=b^n_j-b^{\infty}_j$, we have
\begin{equation}\label{sp2}
\left\{\begin{array}{lll}
\delta^nu_t-\Delta\delta^n u+u^n_j\nabla\delta^n u+\delta^nu\nabla u^{\infty}_j+\nabla (P^n_j-P^{\infty}_j)=b^n_j\nabla \delta^nb+\delta^nb\nabla b^{\infty}_j,\\
\delta^n b_t+u^n_j\nabla\delta^n b+\delta^nu\nabla b^{\infty}_j=b^n_j\nabla\delta^n u+\delta^n b\nabla u^{\infty}_j ,\\
(\delta^nu,\delta^nb)|_{t=0}=(\dot{S}_ju^n_0,\dot{S}_jb^n_0).
\end{array}\right.
\end{equation}
Multiplying both sides of the first equation in (\ref{sp2}) by $\lambda_j$ ($\lambda_j$ is determined later) and applying Lemma \ref{heat} for (\ref{sp2}), we have
\begin{align}\label{sp3}
&\lambda_j(\|\delta^nu\|_{\dot{B}^{\frac{d}{p}-1}_{p,1}}+\|\delta^nu\|_{L^2_t(\dot{B}^{\frac{d}{p}}_{p,1})}+\|\delta^nu\|_{L^1_t(\dot{B}^{\frac{d}{p}+1}_{p,1})})\notag\\
&\leq \lambda_j\|\dot{S}_j(u^{n}_0-u^{\infty}_0)\|_{\dot{B}^{\frac{d}{p}-1}_{p,1}}
+C\lambda_j\|u^n_j,u^{\infty}_j\|_{L^2_t(\dot{B}^{\frac{d}{p}}_{p,1})}\|\delta^nu\|_{L^2_t\dot{B}^{\frac{d}{p}}_{p,1}}
+C\lambda_j\int_{0}^{t}\|b^n_j,b^{\infty}_j\|_{\dot{B}^{\frac{d}{p}}_{p,1}}\|\delta^nb\|_{\dot{B}^{\frac{d}{p}}_{p,1}}ds\notag\\
&\leq \lambda_j\|u^{n}_0-u^{\infty}_0\|_{\dot{B}^{\frac{d}{p}-1}_{p,1}}
+\frac{\lambda_j}{2}\|\delta^nu\|_{L^2_t(\dot{B}^{\frac{d}{p}}_{p,1})}+C\lambda_j\int_{0}^{t}C_{E_0}\|\delta^nb\|_{\dot{B}^{\frac{d}{p}}_{p,1}}ds,
\end{align}
where $\|u^n_j,u^{\infty}_j\|_{L^2_t\dot{B}^{\frac{d}{p}}_{p,1}}\leq 4a\leq \frac{1}{2C}$ by (\ref{xxiugailsp2}). Taking advantage of Lemma \ref{priori estimate}, we get
\begin{align}\label{sp33}
\|\delta^nb\|_{\dot{B}^{\frac{d}{p}}_{p,1}}
&\leq \|\dot{S}_j(b^{n}_0-b^{\infty}_0)\|_{\dot{B}^{\frac{d}{p}}_{p,1}}
+\int_{0}^{t}\|\delta^nu\|_{\dot{B}^{\frac{d}{p}}_{p,1}}\|b^{\infty}_j\|_{\dot{B}^{\frac{d}{p}+1}_{p,1}}
+\|b^n_j\|_{\dot{B}^{\frac{d}{p}}_{p,1}}\|\delta^nu\|_{\dot{B}^{\frac{d}{p}+1}_{p,1}}
+\|u^{\infty}_j\|_{\dot{B}^{\frac{d}{p}+1}_{p,1}}\|\delta^nb\|_{\dot{B}^{\frac{d}{p}}_{p,1}}ds\notag\\
&\leq \|b^{n}_0-b^{\infty}_0\|_{\dot{B}^{\frac{d}{p}}_{p,1}}+C_{E_0,j}(\|\delta^nu\|_{L^1_t(\dot{B}^{\frac{d}{p}+1}_{p,1})}+\|\delta^nu\|_{L^2_t(\dot{B}^{\frac{d}{p}}_{p,1})})
+\int_{0}^{t}\|u^{\infty}_j\|_{\dot{B}^{\frac{d}{p}+1}_{p,1}}\|\delta^nb\|_{\dot{B}^{\frac{d}{p}}_{p,1}}ds.
\end{align}
Combining \eqref{sp3} and \eqref{sp33}, selecting $\lambda_j$ large enough such that $\lambda_j>4(C_{E_0,j}+1)$, for fixed $j$ we obtain that
\begin{align}\label{sp333}
&\frac{\lambda_j}{4}\|\delta^nu\|_{\dot{B}^{\frac{d}{p}-1}_{p,1}\cap L^2_t(\dot{B}^{\frac{d}{p}}_{p,1})\cap L^1_t(\dot{B}^{\frac{d}{p}+1}_{p,1})}+\|\delta^nb\|_{{\dot{B}^{\frac{d}{p}}_{p,1}}}\notag\\
&\leq C_{E_0,j}(\|b^{n}_0-b^{\infty}_0\|_{\dot{B}^{\frac{d}{p}}_{p,1}}+\|u^{n}_0-u^{\infty}_0\|_{\dot{B}^{\frac{d}{p}-1}_{p,1}})
+\int_{0}^{t}C_{E_0,j}\|\delta^nb\|_{\dot{B}^{\frac{d}{p}}_{p,1}}ds
\rightarrow 0,\quad n\rightarrow \infty,
\end{align}
where the last inequality is based on the Gronwall inequality.
 This implies that for any fixed $j$, we have
\begin{align}\label{sp8}
\|u^n_j-u^{\infty}_j\|_{L^{\infty}_{t}(\dot{B}^{\frac{d}{p}-1}_{p,1})\cap L^1_{t}(\dot{B}^{\frac{d}{p}+1}_{p,1})}+\|b^n_j-b^{\infty}_j\|_{L^{\infty}_{t}(\dot{B}^{\frac{d}{p}}_{p,1})} \rightarrow 0,\quad n\rightarrow\infty .
\end{align}
\quad\\
\textbf{Step 2. Estimate }$\|u^n-u^n_j\|_{L^{\infty}_{T}(\dot{B}^{\frac{d}{p}-1}_{p,1})\cap L^1_{T}(\dot{B}^{\frac{d}{p}+1}_{p,1})}$ \textbf{for any $n\in\mathbb{N}\cup \{\infty\}$} .

Letting $\delta_j u=u^n-u^{n}_j$ and $\delta_j b=b^n-b^{n}_j$, then we have
\begin{equation}\label{spp1}
\left\{\begin{array}{lll}
\delta_ju_t-\Delta\delta_j u+u^n\nabla\delta_j u+\delta_ju\nabla u^n_j+\nabla (P^n-P^n_j)=b^n\nabla \delta_jb+\delta_jb\nabla b^n_j,\\
\delta_j b_t+u^n\nabla\delta_j b+\delta_ju\nabla b^n_j=b^n\nabla\delta_j u+\delta_j b\nabla u^n_j ,\\
(\delta_ju_0,\delta_jb_0)|_{t=0}=((Id-S_j)u^n_0,(Id-S_j)b^n_0).
\end{array}\right.
\end{equation}
By Lemma \ref{zhibiao}, for $p\leq 2d$ we have
$
\|fg\|_{\dot{B}^{\frac{d}{p}-1}_{p,\infty}}\leq \|f\|_{\dot{B}^{\frac{d}{p}-1}_{p,\infty}}\|g\|_{\dot{B}^{\frac{d}{p}}_{p,1}}.
$
Using Lemmas \ref{heat}-\ref{priori estimate} to \eqref{spp1}, we have
\begin{align}\label{spp2}
&\|\delta_ju\|_{\dot{B}^{\frac{d}{2}-2}_{p,\infty}}+\|\delta_ju\|_{\widetilde{L}^2_t(\dot{B}^{\frac{d}{p}-1}_{p,\infty})}+\|\delta_ju\|_{\widetilde{L}^1_t(\dot{B}^{\frac{d}{p}}_{p,\infty})}\notag\\
&\leq \|(Id-\dot{S}_j)u^{n}_0\|_{\dot{B}^{\frac{d}{2}-2}_{p,1}}
+C\|u^n_j,u^n\|_{\widetilde{L}^2_t(\dot{B}^{\frac{d}{p}}_{p,1})}\|\delta_j u\|_{\widetilde{L}^2_t(\dot{B}^{\frac{d}{p}-1}_{p,\infty})}
+\int_{0}^{t}\|b^n_j,b^n\|_{\dot{B}^{\frac{d}{p}}_{p,1}}\|\delta_j b\|_{\dot{B}^{\frac{d}{p}-1}_{p,\infty}}ds\notag\\
&\leq 2^{-j}\|(Id-S_j)u^{n}_0\|_{\dot{B}^{\frac{d}{p}-1}_{p,1}}
+\frac{1}{2}\|\delta_j u\|_{\widetilde{L}^2_t(\dot{B}^{\frac{d}{p}-1}_{p,\infty})}
+\int_{0}^{t}C_{E_0}\|\delta_jb\|_{\dot{B}^{\frac{d}{p}-1}_{p,\infty}}ds,
\end{align}
where we used the fact that $\|u^n_j,u^n\|_{L^2_{T}(\dot{B}^{\frac{d}{p}}_{p,1})} \leq 4a\leq\frac{1}{2C}$, and
\begin{align}\label{spp22}
\|\delta_j b\|_{\dot{B}^{\frac{d}{p}-1}_{p,\infty}}
&\leq \|(Id-S_j)b^{n}_0\|_{\dot{B}^{\frac{d}{p}-1}_{p,1}}
+\int_{0}^{t}[\|b^n_j,b^n\|_{\dot{B}^{\frac{d}{p}}_{p,1}}\|\delta_j u\|_{\dot{B}^{\frac{d}{p}}_{p,1}}
+\|u^n_j\|_{\dot{B}^{\frac{d}{p}+1}_{p,1}}\|\delta_jb\|_{\dot{B}^{\frac{d}{p}-1}_{p,\infty}}ds\notag\\
&\leq 2^{-j}\|(Id-S_j)b^{n}_0\|_{\dot{B}^{\frac{d}{p}}_{p,1}}
+C_{E_0}\|\delta_j u\|_{L^1_t(\dot{B}^{\frac{d}{p}}_{p,1})},
\end{align}
where we used the fact that $\|(Id-S_j)v\|_{\dot{B}^{\frac{d}{p}-m}_{p,1}}\leq C\|(Id-S_j)v\|_{\dot{B}^{\frac{d}{p}}_{p,1}}2^{-m}, m>0$, and the last inequality is based on the Gronwall inequality.

By interpolation, it follows that
\begin{align}\label{interpolation}
\|\delta_j u\|_{L^1_t(\dot{B}^{\frac{d}{p}}_{p,1})}
\leq C\|\delta_j u\|_{\widetilde{L}^1_t(\dot{B}^{\frac{d}{p}}_{p,\infty})} ln(e+\frac{\|\delta_j u\|_{L^1_t(\dot{B}^{\frac{d}{p}-1}_{p,1})}+\|\delta_j u\|_{L^1_t(\dot{B}^{\frac{d}{p}+1}_{p,1})}}{\|\delta_j u\|_{\widetilde{L}^1_t(\dot{B}^{\frac{d}{p}}_{p,\infty})}}),
\end{align}
which together with \eqref{spp2} and \eqref{spp22} yields that
\begin{align}\label{sppp2}
&\|\delta_ju\|_{\dot{B}^{\frac{d}{2}-2}_{p,\infty}}+\|\delta_ju\|_{\widetilde{L}^2_t(\dot{B}^{\frac{d}{p}-1}_{p,\infty})}+\|\delta_ju\|_{\widetilde{L}^1_t(\dot{B}^{\frac{d}{p}}_{p,\infty})}\notag\\
&\leq C_{E_0}(\|(Id-S_j)u^{n}_0\|_{\dot{B}^{\frac{d}{p}-1}_{p,1}}+\|(Id-S_j)b^{n}_0\|_{\dot{B}^{\frac{d}{p}}_{p,1}})
+ C_{E_0}\int_{0}^{t}\|\delta_ju\|_{\widetilde{L}^1_s(\dot{B}^{\frac{d}{p}}_{p,\infty})}
ln(e+\frac{C_{E_0}}{\|\delta_ju\|_{\widetilde{L}^1_s(\dot{B}^{\frac{d}{p}}_{p,\infty})}})ds.
\end{align}
By Lemma \ref{osgood} with $\mu(r)=rln(e+\frac{C_{E_0}}{r})$, $\gamma(s)=C_{E_0}$, we obtain
\begin{align}\label{spppp2}
&\|\delta_ju\|_{L^{\infty}_t(\dot{B}^{\frac{d}{2}-2}_{p,\infty})}+\|\delta_ju\|_{\widetilde{L}^2_t(\dot{B}^{\frac{d}{p}-1}_{p,\infty})}+\|\delta_ju\|_{\widetilde{L}^1_t(\dot{B}^{\frac{d}{p}}_{p,\infty})}\notag\\
&\leq C_{E_0}(\|(Id-S_j)u^{n}_0\|_{\dot{B}^{\frac{d}{p}-1}_{p,1}}+\|(Id-S_j)b^{n}_0\|_{\dot{B}^{\frac{d}{p}}_{p,1}})\notag\\
&\rightarrow 0,\quad j\rightarrow\infty,\quad \forall n\in\mathbb{N}\cup\{\infty\} .
\end{align}
Thus, by \eqref{spp22} and \eqref{interpolation} we have
\begin{align}\label{spppp3}
\|\delta_jb\|_{L^{\infty}_t(\dot{B}^{\frac{d}{p}-1}_{p,\infty})},\|\delta_ju\|_{L^1_t(\dot{B}^{\frac{d}{p}}_{p,1})}\rightarrow 0,\quad j\rightarrow\infty,\quad\forall n\in\mathbb{N}\cup\{\infty\}.
\end{align}

Next we estimate $\|\delta_ju\|_{L^{\infty}_{t}(\dot{B}^{\frac{d}{p}-1}_{p,1})\cap (L^1_{t}\dot{B}^{\frac{d}{p}+1}_{p,1})}$. Similarly, we have
\begin{align}\label{spp3}
&\|\delta_ju\|_{L^{\infty}_{t}(\dot{B}^{\frac{d}{p}-1}_{p,1})}
+\|\delta_ju\|_{L^2_{t}(\dot{B}^{\frac{d}{p}}_{p,1})}
+\|\delta_ju\|_{L^1_{t}(\dot{B}^{\frac{d}{p}+1}_{p,1})}\notag\\
&\leq \|(Id-\dot{S}_j)u^n_0\|_{B^{\frac{d}{p}-1}_{p,1}}
+C\|u^n_j,u^n\|_{L^2_t(\dot{B}^{\frac{d}{p}}_{p,1})}\|\delta_ju\|_{L^2_t(\dot{B}^{\frac{d}{p}}_{p,1})}
+C\int_{0}^{t}\|b^n_j\|_{\dot{B}^{\frac{d}{p}}_{p,1}}\|\delta_jb\|_{\dot{B}^{\frac{d}{p}}_{p,1}}ds \notag\\
&\leq  \|(Id-\dot{S}_j)u^n_0\|_{B^{\frac{d}{p}-1}_{p,1}}
+\frac{1}{2}\|\delta_ju\|_{L^2_t(\dot{B}^{\frac{d}{p}}_{p,1})}
+\int_{0}^{t}C_{E_0}\|\delta_jb\|_{\dot{B}^{\frac{d}{p}}_{p,1}}ds,
\end{align}
which implies that
\begin{align}\label{spp33}
\|\delta_ju\|_{L^{\infty}_{t}(\dot{B}^{\frac{d}{p}-1}_{p,1})\cap L^1_{t}(\dot{B}^{\frac{d}{p}+1}_{p,1})}\leq  C\|(Id-\dot{S}_j)u^n_0\|_{B^{\frac{d}{p}-1}_{p,1}}
+C_{E_0}\int_{0}^{t}\|\delta_jb\|_{\dot{B}^{\frac{d}{p}}_{p,1}}ds,\quad \forall n\in\mathbb{N}\cup\{\infty\}.
\end{align}
Thus, we must combine the estimation of (\ref{spp33}) with $\|\delta_jb\|_{\dot{B}^{\frac{d}{p}}_{p,1}}$ to prove the continuous dependence of $(\delta_ju,\delta_jb)$.\\
\quad\\
\textbf{Step 3. Estimate } $\|b^n-b^n_j\|_{L^{\infty}_{t}(\dot{B}^{\frac{d}{p}}_{p,1})}$ \textbf{for any $n\in\mathbb{N}\cup \{\infty\}$} .

Define that $b^{n}_{\infty}:=b^{n}$, $u^{n}_{\infty}:=u^{n}$ and recall the equations of $b^n_j$ with $n,j\in\mathbb{N}\cup\{\infty\}$:
\begin{equation}\label{equ1-1-5}
  \left\{\begin{array}{l}
   \frac{d}{dt}b^n_j+u^n_j\nabla b^n_j=b^n_j\nabla u^n_j,  \\
   b^n_j(0,x)=\dot{S}_jb^n_0.
  \end{array}\right.
\end{equation}
We let $b^n_j:=w^n_j+z^n_j$ such that
\begin{equation}\label{spp5}
\left\{\begin{array}{lll}
\frac{d}{dt}w^n_j+u^n_j\nabla w^n_j=F^{\infty},\\
w^n_j|_{t=0}=b^n_0,
\end{array}\right.
\end{equation}
and
\begin{equation}\label{spp6}
\left\{\begin{array}{lll}
\frac{d}{dt}z^n_j+u^n_j\nabla z^n_j=F^{j}-F^{\infty},\\
z^n_j|_{t=0}=\dot{S}_jb^n_0-b^n_0,
\end{array}\right.
\end{equation}
where $ F^{j}:=b^n_j\nabla u^n_j$
and $F^{\infty}:=b^n_{\infty}\nabla u^n_{\infty}$.

Since $F^{\infty},F^{j}$ are bounded in $L^1_{T}(\dot{B}^{\frac{d}{p}}_{p,1})\cap L^2_{T}(\dot{B}^{\frac{d}{p}-1}_{p,1})$, by Remark \ref{Besov}, we deduce that (\ref{spp5}) and (\ref{spp6}) have a unique solution $w^n_j,z^n_j\in C_{T}(\dot{B}^{\frac{d}{p}}_{p,1})$.

Our main idea is to verify that
$(w^n_j,z^n_j)\rightarrow (w^n_{\infty},0)\text{ in }\dot{B}^{\frac{d}{p}}_{p,1}$ for any $n\in\mathbb{N}\cup \{\infty\}$, which implies that
$ b^n_j\rightarrow b^n_{\infty}\text{ in }\dot{B}^{\frac{d}{p}}_{p,1}.$
For this purpose, we divide the verification into the following three small parts.

Firstly, we estimate $\|w^n_j-w^n_{\infty}\|_{L^{\infty}_{t}(\dot{B}^{\frac{d}{p}}_{p,1})}$.
Similarly to \eqref{frequence}, we see that
\begin{align}\label{sppbuchong1}
\|w^n_j-w^n_{\infty}\|_{L^{\infty}_{T}(\dot{B}^{\frac{d}{p}}_{p,1})}\leq \|w^n_j-w^n_{jk}\|_{L^{\infty}_{T}(\dot{B}^{\frac{d}{p}}_{p,1})}+\|w^n_{jk}-w^n_{\infty k}\|_{L^{\infty}_{T}(\dot{B}^{\frac{d}{p}}_{p,1})}+\|w^n_{\infty k}-w^n_{\infty}\|_{L^{\infty}_{T}(\dot{B}^{\frac{d}{p}}_{p,1})},
\end{align}
where
\begin{equation}\label{wnjk}
\left\{\begin{array}{lll}
\frac{d}{dt}w^n_{j k}+u^n_{j}\nabla(w^n_{j k})=\dot{S}_kF_{\infty},\\
w^n_{jk}|_{t=0}=\dot{S}_kb^n_0.
\end{array}\right.
\end{equation}
\romannumeral1. Estimate $\|w^n_{jk}-w^n_{\infty k}\|_{L^{\infty}_{t}(\dot{B}^{\frac{d}{p}}_{p,1})}$ for fixed $k$ . \\
From (\ref{wnjk}) we deduce that:
\begin{equation}\label{spp7}
\left\{\begin{array}{lll}
\frac{d}{dt}(w^n_{jk}-w^n_{\infty k})+u^n_j\nabla(w^n_{jk}-w^n_{\infty k})=-(u^n_j-u^n_{\infty})\nabla w^n_{\infty k},\\
(w^n_{jk}-w^n_{\infty k})|_{t=0}=0.
\end{array}\right.
\end{equation}
By Lemma \ref{priori estimate} we have
$$\|w^n_{\infty k}\|_{L^{\infty}_{T}(\dot{B}^{\frac{d}{p}+1}_{p,1})}\leq
\|\dot{S}_kb^{\infty}_0\|_{\dot{B}^{\frac{d}{p}+1}_{p,1}}
+\|\dot{S}_kF^{\infty}\|_{L^1_{T}(\dot{B}^{\frac{d}{p}+1}_{p,1})}\leq 2^k\|b^{\infty}_0\|_{\dot{B}^{\frac{d}{p}}_{p,1}}
+2^k\|F^{\infty}\|_{L^1_{T}(\dot{B}^{\frac{d}{p}}_{p,1})}\leq 2^kC_{E_0},$$
and
\begin{align}\label{spp8}
\|w^n_{jk}-w^n_{\infty k}\|_{L^{\infty}_{t}(\dot{B}^{\frac{d}{p}}_{p,1})}&\leq \int_{0}^{t} \|u^n_j-u^n_{\infty}\|_{\dot{B}^{\frac{d}{p}}_{p,1}}\|w^n_{\infty k}\|_{\dot{B}^{\frac{d}{p}+1}_{p,1}}ds\notag\\
&\leq \int_{0}^{t} \|u^n_j-u^n_{\infty}\|_{\dot{B}^{\frac{d}{p}}_{p,1}}(2^kC_{E_0})ds\notag\\
&\leq 2^kC_{E_0}\|u^n_j-u^n_{\infty}\|_{L^1_t(\dot{B}^{\frac{d}{p}}_{p,1})}\notag\\
& \rightarrow 0,\quad j\rightarrow\infty,
\end{align}
where the last inequality is based on (\ref{spppp3}).\\
\romannumeral2. Estimate $\|w^n_j-w^n_{jk}\|_{L^{\infty}_{t}(\dot{B}^{\frac{d}{p}}_{p,1})}$ for any $j\in\mathbb{N}\cup\{\infty\}$.\\
From (\ref{spp5}) and (\ref{wnjk}), we obtain
\begin{equation}\label{spp11}
\left\{\begin{array}{lll}
\frac{d}{dt}(w^n_{j}-w^n_{jk})+u^n_j\nabla(w^n_{j}-w^n_{jk})=(Id-\dot{S}_k)F_{\infty},\\
(w^i_{j}-w^i_{jk})|_{t=0}=(Id-\dot{S}_k)b_0.
\end{array}\right.
\end{equation}
By Lemma \ref{priori estimate}, we have
\begin{align}\label{spp12}
\|w^n_{j}-w^n_{jk}\|_{L^{\infty}_{t}(\dot{B}^{\frac{d}{p}}_{p,1})}&\leq \|(Id-\dot{S}_k)b_0\|_{\dot{B}^{\frac{d}{p}}_{p,1}}+\int_{0}^{t}\|(Id-\dot{S}_k)F_{\infty}\|_{\dot{B}^{\frac{d}{p}}_{p,1}}ds    \notag\\
&\rightarrow 0,\quad k\rightarrow\infty .
\end{align}

 By (\ref{spp12}), for any $\epsilon>0$, there exists $N$ (independent on $j$ and $n$) such that when $k\geq N$ we have
$$\|w^n_j-w^n_{jk}\|_{L^\infty_t(\dot{B}^{\frac{d}{p}}_{p,1})}\leq\frac{\epsilon}{3},\quad
\|w^n_{\infty}-w^n_{\infty k}\|_{L^\infty_t(\dot{B}^{\frac{d}{p}}_{p,1})}\leq\frac{\epsilon}{3}.$$
For this $\epsilon$ and $k\geq N$, by (\ref{spp8}), there exists $M$ (independent on $n$) such that when $j\geq M$ we have
$$\|w^n_{jk}-w^n_{\infty k}\|_{L^\infty_t(\dot{B}^{\frac{d}{p}}_{p,1})}\leq\frac{\epsilon}{3}.$$
Thus we get
\begin{align}
\|w^n_j-w^n_{\infty}\|_{L^\infty_t(\dot{B}^{\frac{d}{p}}_{p,1})}\leq \|w^n_j-w^n_{jk}\|_{L^\infty_t(\dot{B}^{\frac{d}{p}}_{p,1})}+\|w^n_{\infty}-w^n_{\infty k}\|_{L^\infty_t(\dot{B}^{\frac{d}{p}}_{p,1})}+\|w^n_{jk}-w^n_{\infty k}\|_{L^\infty_t(\dot{B}^{\frac{d}{p}}_{p,1})}\leq\epsilon ,
\end{align}
that is
\begin{align}\label{spp14}
\|w^n_j-w^n_{\infty}\|_{L^{\infty}_{t}(\dot{B}^{\frac{d}{p}}_{p,1})}\rightarrow 0,\quad j\rightarrow\infty ,\quad\forall n\in \mathbb{N}\cup \{\infty\} .
\end{align}

Next, we estimate $\|z^n_j\|_{L^{\infty}_{T}(\dot{B}^{\frac{d}{p}}_{p,1})}$. Recall that
\begin{equation}\label{spp15}
\left\{\begin{array}{lll}
\frac{d}{dt}z^n_j+u^n_j\nabla z^n_j=F^{j}-F^{\infty},\\
z^i_j|_{t=0}=(\dot{S}_j-Id)b_0,
\end{array}\right.
\end{equation}
where $F^{j}:=b^n_j\nabla u^n_j$. By the Bony decomposition, we have
\begin{align}\label{spp16}
&\|F^{j}-F^{\infty}\|_{\dot{B}^{\frac{d}{p}}_{p,1}} \notag\\
&\leq \|(b^n_j-b^n_{\infty})\nabla u^n_j\|_{\dot{B}^{\frac{d}{p}}_{p,1}} +\|b^n_{\infty}\nabla (u^n_j-u^n_{\infty})\|_{\dot{B}^{\frac{d}{p}}_{p,1}} \notag\\
&\leq \|b^n_j-b^n_{\infty}\|_{\dot{B}^{\frac{d}{p}}_{p,1}}\|u^n_j\|_{\dot{B}^{\frac{d}{p}+1}_{p,1}}+ \|u^n_j-u^n_{\infty}\|_{\dot{B}^{\frac{d}{p}+1}_{p,1}}\|b^n_{\infty}\|_{\dot{B}^{\frac{d}{p}}_{p,1}}  \notag\\
&\leq (\|z^n_j\|_{\dot{B}^{\frac{d}{p}}_{p,1}}+\|w^n_j-w^n_{\infty}\|_{\dot{B}^{\frac{d}{p}}_{p,1}})\|u^n_j\|_{\dot{B}^{\frac{d}{p}+1}_{p,1}}
+\|u^n_j-u^n_{\infty}\|_{\dot{B}^{\frac{d}{p}+1}_{p,1}}\|b^n_{\infty}\|_{\dot{B}^{\frac{d}{p}}_{p,1}},
\end{align}
where the last inequality is based on $\|b^n_j-b^n_{\infty}\|_{\dot{B}^{\frac{d}{p}}_{p,1}}\leq \|w^n_j-w^n_{\infty}\|_{\dot{B}^{\frac{d}{p}}_{p,1}}+\|z^n_j\|_{\dot{B}^{\frac{d}{p}}_{p,1}}$. Combining \eqref{spp33} , \eqref{spp16} and \eqref{spp15}, we have
\begin{align}
&\|z^n_j\|_{\dot{B}^{\frac{d}{p}}_{p,1}} \notag\\
\leq &  \|(Id-\dot{S}_j)b^n_0\|_{\dot{B}^{\frac{d}{p}}_{p,1}}+C\int_{0}^{t}(\|z^n_j\|_{\dot{B}^{\frac{d}{p}}_{p,1}}+\|w^n_j-w^n_{\infty}\|_{\dot{B}^{\frac{d}{p}}_{p,1}})\|u^n_j\|_{\dot{B}^{\frac{d}{p}+1}_{p,1}}ds
+C_{E_0}\|u^n_j-u^n_{\infty}\|_{L^1_t\dot{B}^{\frac{d}{p}+1}_{p,1}} \notag\\
\leq &  C_{E_0}(\|(Id-\dot{S}_j)b^n_0\|_{\dot{B}^{\frac{d}{p}}_{p,1}}+\|(Id-S_j)u^n_0\|_{\dot{B}^{\frac{d}{p}-1}_{p,1}}
+\|w^n_j-w^n_{\infty}\|_{L^{\infty}_{t}(\dot{B}^{\frac{d}{p}}_{p,1})})(C_{E_0}+\|u^n_j\|_{L^1_t(\dot{B}^{\frac{d}{p}+1}_{p,1})})\notag\\
&+\int_{0}^{t}C_{E_0}(\|u^n_j\|_{\dot{B}^{\frac{d}{p}+1}_{p,1}}+1)\|z^n_j\|_{\dot{B}^{\frac{d}{p}}_{p,1}}ds.
\end{align}
Applying the Gronwall inequality and \eqref{spp14}, we obtain
\begin{align}\label{spp18}
\|z^n_j\|_{L^{\infty}_{t}(\dot{B}^{\frac{d}{p}}_{p,1})}
\leq &  C_{E_0}(\|(Id-S_j)b^n_0\|_{\dot{B}^{\frac{d}{p}}_{p,1}}+\|(Id-S_j)u^n_0\|_{\dot{B}^{\frac{d}{p}-1}_{p,1}}+
\|w^n_j-w^n_{\infty}\|_{\dot{B}^{\frac{d}{p}}_{p,1}}) \notag\\
\rightarrow & 0,\quad j\rightarrow\infty,\quad\forall n\in \mathbb{N}\cup \{\infty\} .
\end{align}

Finally, combining \eqref{spp14} and \eqref{spp18}, we have
\begin{align}\label{spp19}
\|b^n_j-b^n_{\infty}\|_{L^{\infty}_{t}(\dot{B}^{\frac{d}{p}}_{p,1})}\rightarrow 0,\quad j\rightarrow\infty,\quad \forall n \in \mathbb{N}\cup \{\infty\},
\end{align}
and
\begin{align}\label{spp1919}
&\|u^n_j-u^n_{\infty}\|_{L^{\infty}_{t}(\dot{B}^{\frac{d}{p}-1}_{p,1})\cap L^1_{t}(\dot{B}^{\frac{d}{p}+1}_{p,1})}\notag\\
&\leq  C\|(Id-\dot{S}_j)u^n_0\|_{B^{\frac{d}{p}-1}_{p,1}}
+C_{E_0}\int_{0}^{t}\|b^n_j-b^n_{\infty}\|_{\dot{B}^{\frac{d}{p}}_{p,1}}ds\notag\\
&\rightarrow 0,\quad j\rightarrow\infty,\quad \forall n\in \mathbb{N}\cup \{\infty\}.
\end{align}
Thus, we complete the estimations of $\|b^n_j-b^n_{\infty}\|_{L^{\infty}_{t}(\dot{B}^{\frac{d}{p}}_{p,1})}$  and $\|u^n_j-u^n_{\infty}\|_{L^{\infty}_{t}(\dot{B}^{\frac{d}{p}-1}_{p,1})\cap L^1_{t}(\dot{B}^{\frac{d}{p}+1}_{p,1})}$.\\
\quad\\
\textbf{Step 4. Proof of the continuous dependence}

Finally, combining (\ref{spp1919}) and (\ref{spp19}), we obtain
\begin{align}\label{spp20}
\|u^n_j-u^n\|_{L^{\infty}_{t}(\dot{B}^{\frac{d}{p}-1}_{p,1})\cap L^1_{t}(\dot{B}^{\frac{d}{p}+1}_{p,1})}
+\|b^n_j-b^n_{\infty}\|_{L^{\infty}_{t}(\dot{B}^{\frac{d}{p}}_{p,1})}\rightarrow 0\quad,j\rightarrow\infty,\forall n\in \mathbb{N}\cup \{\infty\}.
\end{align}
By (\ref{spp20}), for any $ \epsilon>0$, there exists $N$ (independent of $n$) such that when $j\geq N$ we have
$$\|u^n-u^n_j\|_{L^{\infty}_{t}(\dot{B}^{\frac{d}{p}-1}_{p,1})\cap L^1_{t}(\dot{B}^{\frac{d}{p}+1}_{p,1})}\leq\frac{\epsilon}{3}\quad \forall n\in \mathbb{N}^+\cup \{\infty\}.$$
For this $\epsilon$ and $j\geq N$, by (\ref{sp8}) there exists $M$ such that when $n\geq M$, we get
$$\|u^n_j-u^{\infty}_j\|_{L^{\infty}_{t}(\dot{B}^{\frac{d}{p}-1}_{p,1})\cap L^1_{t}(\dot{B}^{\frac{d}{p}+1}_{p,1})}\leq\frac{\epsilon}{3}.$$
Thus we deduce
\begin{align}
&\quad\|u^n-u^{\infty}\|_{L^{\infty}_{t}(\dot{B}^{\frac{d}{p}-1}_{p,1}\cap L^1_{t}\dot{B}^{\frac{d}{p}+1}_{p,1})}\notag\\
&\leq\|u^n-u^n_j\|_{L^{\infty}_{t}(\dot{B}^{\frac{d}{p}-1}_{p,1})\cap L^1_{t}(\dot{B}^{\frac{d}{p}+1}_{p,1})}+\|u^n_j-u^{\infty}_j\|_{L^{\infty}_{t}(\dot{B}^{\frac{d}{p}-1}_{p,1})\cap L^1_{t}(\dot{B}^{\frac{d}{p}+1}_{p,1})}+\|u^{\infty}_j-u^{\infty}\|_{L^{\infty}_{t}(\dot{B}^{\frac{d}{p}-1}_{p,1}\cap L^1_{t}\dot{B}^{\frac{d}{p}+1}_{p,1})}\notag\\
&\leq \epsilon.
\end{align}
Similarly, we have
\begin{align}
\|b^n-b^{\infty}\|_{L^{\infty}_{t}(\dot{B}^{\frac{d}{p}}_{p,1})}
\leq\|b^n-b^n_j\|_{L^{\infty}_{t}(\dot{B}^{\frac{d}{p}}_{p,1})}
+\|b^n_j-b^{\infty}_j\|_{L^{\infty}_{t}(\dot{B}^{\frac{d}{p}}_{p,1})}
+\|b^{\infty}_j-b^{\infty}\|_{L^{\infty}_{t}(\dot{B}^{\frac{d}{p}}_{p,1})}\leq \epsilon.
\end{align}

This completes the proof of the continuous dependence in $t\in$ $[0,T]$.
\end{proof}

\noindent\textbf{Acknowledgements.} This work was partially supported by National Natural Science Foundation
of China [grant number 11671407 and 11701586], the Macao Science and Technology Development Fund (grant number 0091/2018/A3), Guangdong Special Support Program (grant number 8-2015), and the key project of NSF of  Guangdong province (grant number 2016A030311004).

\phantomsection
\addcontentsline{toc}{section}{\refname}

\bibliographystyle{abbrv} 
\bibliography{Ye-Luo-Yinref}

\end{document}